\pdfoutput=1
\documentclass{article}
\usepackage[a4paper]{geometry}
\usepackage[english]{babel,isodate}
\usepackage{amsmath,amssymb,bm}
\usepackage{amsthm,thmtools}
\usepackage{enumitem}
\usepackage{lmodern}
\usepackage{microtype}
\usepackage{xcolor}
\usepackage{hyperref}
\usepackage[nameinlink,capitalise,noabbrev]{cleveref} % Some sources say this package should be loaded after hyperref, but I don't see this mentioned in the documentation. Loading it here just to be safe.

\title{On the size of subsets of \texorpdfstring{$\F_q^n$}{\textbackslash{}F\textunderscore{}q\textasciicircum{}n} avoiding solutions\texorpdfstring{\\}{} to linear systems with repeated columns}
\author{Josse van Dobben de Bruyn \texorpdfstring{\and}{and} Dion Gijswijt}
\date{22 September 2023}

\cleanlookdateon

\makeatletter
\newcommand{\my@author}[1]{\par\bigskip\noindent #1}
\newcommand{\my@address}[1]{\\\textsc{#1}}
\newcommand{\my@currentaddress}[1]{\\\textit{Current address:} #1}
\newcommand{\my@email}[1]{\\\textit{E-mail address:} \texttt{\href{mailto:#1}{#1}}}
\AtEndDocument{
	\interlinepenalty9999
	\par\bigskip
	\my@author{Josse van Dobben de Bruyn}
	\my@address{Delft Institute of Applied Mathematics, Delft University of Technology, Mekelweg~4, 2628~CD~Delft, The Netherlands.}
	\my@currentaddress{DTU Compute, Technical University of Denmark, Richard~Petersens~Plads, Building~322, 2800~Kongens~Lyngby, Denmark.}
	\my@email{jdob@dtu.dk}
	
	\my@author{Dion Gijswijt}
	\my@address{Delft Institute of Applied Mathematics, Delft University of Technology, Mekelweg~4, 2628~CD~Delft, The Netherlands.}
	\my@email{D.C.Gijswijt@tudelft.nl}
}
\makeatother

\newcommand{\hair}{\ifmmode\mskip1mu\else\kern0.08em\fi} % from http://www.read.seas.harvard.edu/~kohler/latex.html

\newcommand{\savefootnote}[2]{%
	\xdef#1{\thefootnote}%
	\footnote{#2}%
}
\newcommand{\reusefootnotemark}[1]{%
	\edef\curfootnote{\thefootnote}%
	\edef\curHfootnote{\theHfootnote}%
	\setcounter{footnote}{#1}%
	\setcounter{Hfootnote}{#1}%            second footnote counter for use by hyperref; see e.g. https://tex.stackexchange.com/a/35045
	\footnotemark{}%
	\setcounter{footnote}{\curfootnote}%
	\setcounter{Hfootnote}{\curHfootnote}% second footnote counter for use by hyperref; see e.g. https://tex.stackexchange.com/a/35045
}

\makeatletter
\def\my@linkcolour{blue!60!black}
\hypersetup{pdfauthor={\@author},pdftitle={\@title},colorlinks=true,linkcolor=\my@linkcolour,urlcolor=\my@linkcolour,citecolor=\my@linkcolour}
\makeatother

\newcommand*{\mycref}[2]{\hyperref[#2]{\cref*{#1}\ref*{#2}}}  % \mycref{prop:1}{itm:a} creates a link [Proposition 1(a)].
\newcommand*{\MyCref}[2]{\hyperref[#2]{\Cref*{#1}\ref*{#2}}}  % same, but making sure the first letter is capitalized (regardless of the options passed to cleveref)
\newcommand*{\myautorefstar}[2]{\texorpdfstring{\mycref{#1}{#2}}{\autoref{#1}\ref{#2}}} % a version of the previous command that can be used in section titles (CAUTION: not fully compatible with e-jc package; only use with thmtools-defined theorem styles) (in this document, the aforementioned condition is met -- it is only used with AlphTheorem)

\newcommand{\mysecref}[1]{\hyperref[#1]{\S\ref*{#1}}} % manual autoref for section (to suppress superfluous space)

\newcommand{\exampleqed}{\ensuremath{\scriptstyle\triangle}}

\numberwithin{equation}{section}

\declaretheorem[style=definition,numberlike=equation]{definition}
\declaretheorem[style=definition,numberlike=definition]{situation}
\declaretheorem[style=definition,qed=\exampleqed,sibling=definition]{example}

\declaretheorem[style=plain,numberlike=definition]{theorem}
\declaretheorem[style=plain,numberlike=definition]{lemma}
\declaretheorem[style=plain,numberlike=definition]{proposition}
\declaretheorem[style=plain,numberlike=definition]{corollary}

\declaretheorem[style=remark,numberlike=definition]{remark}

\declaretheorem[style=plain,title={Theorem}]{AlphTheorem}

\declaretheorem[style=plain,title={Corollary},numberlike=AlphTheorem]{AlphCorollary}

\newcommand{\Z}{\mathbb{Z}}
\newcommand{\F}{\mathbb{F}}
\newcommand{\E}{\mathbb{E}}

\DeclareSymbolFont{bbold}{U}{bbold}{m}{n}
\DeclareSymbolFontAlphabet{\mathbbold}{bbold}
\newcommand{\one}{\ensuremath{\mathbbold{1}}}

\newcommand{\bigOh}{\ensuremath{\mathcal O}}

\newcommand{\mytag}{\ensuremath{(\star)}}
\newcommand{\mytagprime}{\ensuremath{(\star')}}

\newcommand{\myvec}[1]{\bm{#1}}
\newcommand{\myvecvec}[1]{\bm{\vec{#1}}}
\newcommand{\myvecvecprime}[1]{\bm{\vec{#1}\, '}}
\newcommand{\myvecvecnum}[2]{\bm{\vec{#1}^{\,#2}}}

\DeclareMathOperator{\rank}{rank}
\DeclareMathOperator{\spn}{span}
\DeclareMathOperator{\aff}{aff}
\DeclareMathOperator{\Ann}{Ann}
\DeclareMathOperator{\bal}{bal}
\newcommand{\Annbal}{\ensuremath{\Ann_{\bal}}}
\DeclareMathOperator*{\argmax}{arg\,max}
\DeclareMathOperator*{\rowspace}{row space}

\makeatletter
\DeclareMathOperator{\@aff}{aff}
\newcommand{\AIRplus}{\underset{\@aff}{\dotplus}}
\makeatother

\newcommand{\parteq}{\equiv}
\newcommand{\nparteq}{\not\equiv}
\newcommand{\coleq}{\parallel}
\newcommand{\ncoleq}{\nparallel}

\begin{document}
\maketitle
\begin{abstract}
	Consider a system of $m$ balanced linear equations in $k$ variables with coefficients in $\F_q$.
	If $k \geq 2m + 1$, then a routine application of the slice rank method shows that there are constants $\beta,\gamma \geq 1$ with $\gamma < q$ such that, for every subset $S \subseteq \F_q^n$ of size at least $\beta \cdot \gamma^n$, the system has a solution $(x_1,\ldots,x_k) \in S^k$ with $x_1,\ldots,x_k$ not all equal.
	Building on a series of papers by Mimura and Tokushige and on a paper by Sauermann, this paper investigates the problem of finding a solution of higher non-degeneracy; that is, a solution where $x_1,\ldots,x_k$ are pairwise distinct, or even a solution where $x_1,\ldots,x_k$ do not satisfy any balanced linear equation that is not a linear combination of the equations in the system.
	
	In this paper, we focus on linear systems with repeated columns.
	For a large class of systems of this type, we prove that there are constants $\beta,\gamma \geq 1$ with $\gamma < q$ such that every subset $S \subseteq \F_q^n$ of size at least $\beta \cdot \gamma^n$ contains a solution that is non-degenerate (in one of the two senses described above).
	This class is disjoint from the class covered by Sauermann's result, and captures the systems studied by Mimura and Tokushige into a single proof.
	Moreover, a special case of our results shows that, if $S \subseteq \F_p^n$ is a subset such that $S - S$ does not contain a non-trivial $k$-term arithmetic progression (with $p$ prime and $3 \leq k \leq p$), then $S$ must have exponentially small density.
\end{abstract}

\section{Introduction}
\subsection{Background and prior results}
For several decades, one of the major open problems in extremal combinatorics had been to determine whether or not, for a given prime $p\geq 3$, there is a constant $c_p < p$ such that every subset $S\subseteq \F_p^n$ of size $|S|\geq c_p^n$ contains a non-trivial 3-term arithmetic progression; that is, a solution to the equation $x - 2y + z = 0$ with $x,y,z \in S$ distinct.
For $p = 3$, this problem was known as the \emph{cap set problem}.

In 2016, Ellenberg and Gijswijt~\cite{Ellenberg-Gijswijt} solved this problem for all primes $p \geq 3$, using a new application of the polynomial method developed by Croot, Lev and Pach \cite{Croot-Lev-Pach}.
The solution was subsequently recast by Tao \cite{Tao-slice-rank} in terms of the slice rank of tensors (or hypermatrices).
Together, these developments have led to a surge of interest in problems related to the cap set problem, using the slice rank polynomial method to attempt to solve other problems.

One of these related open problems is to find the largest size of a subset of $\F_p^n$ without a non-trivial $k$-term arithmetic progression, where $p \geq k \geq 4$ and $n \to \infty$.
It is not known whether or not there is a constant $c_{p,k} < p$ such that every set $S \subseteq \F_p^n$ with $|S| \geq c_{p,k}^n$ contains a $k$-term arithmetic progression.
This problem is believed to be beyond the reach of current slice rank methods.

Instead, mathematicians have turned their attention to related problems.
Recently, Mimura and Tokushige \cite{Mimura-Tokushige-star,Mimura-Tokushige-shape,Mimura-Tokushige-II} and Sauermann \cite{Sauermann-systems} have started developing techniques to bound the maximum size of a subset of $\F_q^n$ which avoids non-degenerate solutions to a given system of linear equations over a finite field $\F_q$.
More formally, given a fixed matrix $A = (a_{ij}) \in \F_q^{m \times k}$, we want to bound the maximum size of a subset $S \subseteq \F_q^n$ for which there are no $k$-tuples $(\myvec{x_1},\ldots,\myvec{x_k}) \in S^k$ satisfying
\[ \begin{cases}
	a_{11}\myvec{x_1} + \cdots + a_{1k}\myvec{x_k} = 0,\\[1ex]
	\hspace*{14.75mm}\vdots\\[1ex]
	a_{m1}\myvec{x_1} + \cdots + a_{mk}\myvec{x_k} = 0;
\end{cases}\tag*{\mytag} \]
except possibly trivial/degenerate solutions (more on that later).
Note that the variables $\myvec{x_1},\ldots,\myvec{x_k}$ are not taken from $\F_q$, but from $\F_q^n$ as $n \to \infty$.

If $a_{i1} + \cdots + a_{ik} \neq 0$ for some $i$ (i.e.{} the coefficients in one of the rows do not sum to zero), then there are large subsets of $\F_q^n$ with no solutions at all to \mytag.
Indeed, let $S \subseteq \F_q^n$ be the set of all vectors whose first coordinate is equal to $1$.
If some row of \mytag{} does not sum to zero, then $S$ does not contain solutions to \mytag, and $|S| = q^{n-1} = \frac{1}{q}\cdot |\F_q^n|$, so $S$ contains a constant proportion of the vectors in $\F_q^n$.
(This example is due to Sauermann \cite{Sauermann-systems}.)

We will henceforth assume that $a_{i1} + \cdots + a_{ik} = 0$ for all $i$.
Such equations are called \emph{balanced linear equations} (or \emph{affine dependences}), and the system \mytag{} is also called \emph{balanced}.
Recent results show that the problem becomes much more interesting in this case.

If the system \mytag{} is balanced, then every set $S \subseteq \F_q^n$ has at least $|S|$ solutions to \mytag, namely the solutions of the form $(a,\ldots,a)$ for $a \in S$.
So the question is: how large does $S$ have to be to guarantee the existence of solutions to \mytag{} which are somehow non-degenerate?
For this we consider three different notions of non-degeneracy:
\begin{definition}
	\label{def:non-degenerate}
	A solution $(\myvec{x_1},\ldots,\myvec{x_k}) \in (\F_q^n)^k$ of \mytag{} is called:
	\begin{enumerate}[label=(\alph*),beginpenalty=100,midpenalty=100]
		\item\label{itm:sol:non-trivial} \emph{non-trivial} if $\myvec{x_1},\ldots,\myvec{x_k}$ are not all equal.
		
		\item\label{itm:sol:shape} a \mytag{}-\emph{shape}\hair\savefootnote{\followingMT}{Following terminology from Mimura and Tokushige \cite{Mimura-Tokushige-star,Mimura-Tokushige-shape,Mimura-Tokushige-II}.} if $\myvec{x_1},\ldots,\myvec{x_k}$ are pairwise distinct.
		
		\item\label{itm:sol:generic} \emph{generic}\hair\savefootnote{\introducedByAuthors}{Terminology introduced by the authors.} if every balanced linear equation (over $\F_q$) satisfied by $(\myvec{x_1},\ldots,\myvec{x_k})$ is a linear combination of the equations in \mytag.
	\end{enumerate}
\end{definition}
The requirements get stronger in each step, moving from \ref{itm:sol:non-trivial} to \ref{itm:sol:generic}.
Indeed, it is clear that every \mytag{}-shape is a non-trivial solution.
Furthermore, if the system \mytag{} does not rule out the existence of \mytag{}-shapes in $\F_q^n$ (in other words, if no linear combination of the equations in \mytag{} equals $\myvec{x_i} -\myvec{x_j}=0$ for some $i \neq j$), then every generic solution is a \mytag{}-shape.

The easiest of these problems is finding a non-trivial solution.
If the number of variables is sufficiently large (specifically, if $k \geq 2m + 1$), then this can be done by a routine application of the slice rank method.

\begin{theorem}[{\cite{Tao-slice-rank}, see also \cite[Theorem 1.1]{Sauermann-systems}\hair\protect\footnote{To get rid of the constant factor $C_{q,m,k}$ from \cite[Theorem 1.1]{Sauermann-systems}, use the power trick.}}]
	\label{thm:intro:Tao}
	If $k \geq 2m + 1$, then there exists a constant $\Gamma_{q,m,k} < q$ such that every subset $S \subseteq \F_q^n$ of size at least $(\Gamma_{q,m,k})^n$ has a non-trivial solution of \mytag.
\end{theorem}

If $k \leq 2m$, then the problem of finding non-trivial bounds is believed to be beyond the reach of current (slice rank) methods. Accordingly, most\hair\footnote{The only exception is when we study different, related problems for which the assumption $k \geq 2m + 1$ is not necessary (such as \cref{lem:matrix-pigeonhole}, which is one of the main tools in our proofs).} of our results are on systems with $k \geq 2m + 1$; see \mycref{rem:AboutSit1.7}{rem:AboutSit1.7Part1}.

The aim of this paper is to refine \cref{thm:intro:Tao} to the stronger notions of non-degeneracy from \cref{def:non-degenerate}.
For this we use the following terminology:

\begin{definition}
	\label{def:moderate-temperate}
	The linear system \mytag{} is called:
	\begin{enumerate}[label=(\alph*),beginpenalty=100,midpenalty=100]
		\item \emph{moderate}\hair\reusefootnotemark{\followingMT} if there exist constants $\beta,\gamma > 0$ with $\gamma < q$ such that every subset $S \subseteq \F_q^n$ of size at least $\beta \cdot \gamma^n$ contains a \mytag{}-shape;
		
		\item \emph{temperate}\hair\reusefootnotemark{\introducedByAuthors} if there exist constants $\beta,\gamma > 0$ with $\gamma < q$ such that every subset $S \subseteq \F_q^n$ of size at least $\beta \cdot \gamma^n$ contains a generic solution of \mytag{}.
	\end{enumerate}
\end{definition}

If \mytag{} consists of the single equation $\myvec{x_1} + \cdots + \myvec{x_p} = 0$ over $\F_p$ (with $p$ prime), then the existence of \mytag{}-shapes is tightly linked to the Erd\H{o}s--Ginzburg--Ziv constant of the group $\F_p^n$.
If $p \geq 3$, then this system is moderate over $\F_p$; this is implicit in~\cite{Naslund-EGZ-constant} and \cite{Sauermann-zero}.
Furthermore, the method in \cite{Sauermann-zero} can be easily adapted to show that every balanced linear equation with at least $3$ variables forms a moderate linear system.

The problem of determining whether or not a system of two or more equations is moderate was first studied by Mimura and Tokushige \cite{Mimura-Tokushige-star,Mimura-Tokushige-shape,Mimura-Tokushige-II}.\hair\footnote{Similar results over the integers had been obtained by Ruzsa in the 1990s \cite{Ruzsa-I,Ruzsa-II}, but Mimura and Tokushige were the first to study this problem for vector spaces over a finite field.}
They showed that several specific linear systems are moderate.
Although all of their proofs rely on more or less the same idea, the details of the proofs are so different that a new proof is needed for each new system.
We discuss some of their results in more detail in \mysecref{sec:examples}.

The first general result in this direction was found by Sauermann \cite{Sauermann-systems}.
In an elaborate proof, using a new application of the slice rank method and a subspace sampling argument, she showed that \mytag{}-shapes can always be found if the number of variables is sufficiently large and if the system is very much non-degenerate:

\begin{theorem}[{\cite[Theorem 1.2]{Sauermann-systems}}]
	\label{thm:intro:Sauermann-invertible}
	If $k \geq 3m$ and every $m \times m$ submatrix of $A$ is invertible, then \mytag{} is moderate.
\end{theorem}

Despite its generality, this result does not replace the results of Mimura and Tokushige, because the systems they studied have many singular $m \times m$ submatrices (so \cref{thm:intro:Sauermann-invertible} does not apply).

The third and final problem is that of finding a generic solution.
A partial result in this direction was found by Sauermann, who showed that solutions of higher dimension exist as the number of variables becomes larger:

\begin{theorem}[{\cite[Theorem 1.3]{Sauermann-systems}}]
	\label{thm:intro:Sauermann-rank}
	If $r \geq 2$ and $k \geq 2m - 1 + r$, then there are constants $C_{p,m,k,r}^{\rank} \geq 1$ and $\Gamma_{p,m,k,r}^{\rank} < p$ such that every subset $S \subseteq \F_p^n$ of size at least $C_{p,m,k,r}^{\rank} \cdot (\Gamma_{p,m,k,r}^{\rank})^n$ has a solution $(\myvec{x_1},\ldots,\myvec{x_k}) \in S^k$ of \mytag{} satisfying $\dim(\spn(\myvec{x_1},\ldots,\myvec{x_k})) \geq r$.
\end{theorem}

Finding solutions of high dimension is closely related to finding a generic solution, as we explain in \mysecref{sec:temperate-prelims}.

\subsection{Main results of this paper}
The main results of this paper are twofold.
First, we prove a general result on finding \mytag{}-shapes, which contains most of the results from \cite{Mimura-Tokushige-star,Mimura-Tokushige-shape,Mimura-Tokushige-II} as special cases.
Second, we prove a general result for finding generic solutions, which we believe to be the first of its kind.

Throughout the paper, we focus on a specific class of systems that is completely different from the class of systems studied by Sauermann.
Where Sauermann's result (\cref{thm:intro:Sauermann-invertible} above) requires every $m \times m$ submatrix to be invertible, we require the opposite: there must be sufficiently many linear dependencies between the columns.
Specifically, we focus on the class of `type (RC)' linear systems, which we define as follows:

\begin{definition}
	\label{def:equivalence}
	Consider the linear system \mytag, whose coefficients are specified by the matrix $A = (a_{ij}) \in \F_q^{m \times k}$.
	\begin{enumerate}[label=(\alph*),beginpenalty=100,midpenalty=100]
		\item\label{itm:equiv:indices} We say that two indices in $[k]$ are \emph{equivalent} if the corresponding columns of $A$ are nonzero scalar multiples of one another. This defines an equivalence relation on $[k]$. We will refer to the equivalence classes of this equivalence relation as the \emph{column equivalence classes}.
		
		\item\label{itm:equiv:type-DC} We say that \mytag{} is \emph{a type (RC) linear system}\hair\footnote{Terminology introduced by the authors (`RC' stands for `repeated columns').} if it is balanced and has at most one column equivalence class of size $1$.
		
		\item\label{itm:equiv:sums-to-zero} We say that a column equivalence class \emph{sums to zero} if the columns indexed by that class add up to the zero vector.
	\end{enumerate}
	
	Examples of type (RC) linear systems will be given in \mysecref{sec:examples} below.
	Among these examples are the systems studied by Mimura and Tokushige.
\end{definition}

The assumptions made throughout this paper can be summarized as follows:

\begin{situation}
	\label{sit:system}
	Let \mytag{} be a type (RC) linear system, given by the coefficient matrix $A=(a_{ij}) \in \F_q^{m \times k}$, with $\ell$ column equivalence classes.
	Furthermore, assume that \mytag{} is non-degenerate and irreducible (see \cref{def:systems} below).
\end{situation}

In all of our main results below, we assume that \mytag{} and $A$ are as in \cref{sit:system}.
In particular, we always assume that \mytag{} is irreducible.
However, we note that our results can also be applied to reducible systems.
We show in \cref{prop:moderate-reduce} (resp.{} \cref{prop:temperate-reduce}) that a system is moderate (resp.{} temperate) if and only if every irreducible subsystem is moderate (resp.{} temperate).

Our first main result is a sufficient condition for a type (RC) linear sytem to be moderate.

\begin{AlphTheorem}
	\label{thm:intro:moderate}
	Let \mytag{}, $A$, $m$, $k$ and $\ell$ be as in \cref{sit:system}.
	Suppose that \mytag{} satisfies at least one of the following additional properties:
	\begin{enumerate}[label=\textup(\roman*\textup),beginpenalty=100,endpenalty=100]
		\item\label{itm:intro:mod:nz} none of the column equivalence classes of size $2$ sums to zero;
		\item\label{itm:intro:mod:z} every column equivalence class sums to zero, and $k \geq 3$.
	\end{enumerate}
	Then \mytag{} is moderate.
\end{AlphTheorem}

This result encompasses most of the systems studied by Mimura and Tokushige, and the rest can be recovered using a slight modification of our proof.
See \mysecref{sec:examples} for a detailed discussion.

Our second main result is a sufficient condition for a type (RC) linear sytem to be temperate.

\begin{AlphTheorem}
	\label{thm:intro:temperate}
	Let \mytag{}, $A$, $m$, $k$ and $\ell$ be as in \cref{sit:system}.
	Suppose that \mytag{} satisfies at least one of the following additional properties:
	\begin{enumerate}[label=\textup(\roman*\textup),beginpenalty=100,endpenalty=100]
		\item\label{itm:intro:tmp:nz} $\ell = m + 1$;
		\item\label{itm:intro:tmp:z} every column equivalence class sums to zero.
	\end{enumerate}
	Then \mytag{} is temperate.
\end{AlphTheorem}

The requirements of \cref{thm:intro:temperate} are more restrictive than those of \cref{thm:intro:moderate} (see \mycref{rem:AboutSit1.7}{rem:AboutSit1.7Part2}).\hair\footnote{Except that \mycref{thm:intro:temperate}{itm:intro:tmp:z} does not have the condition $k \geq 3$. That condition is included in \cref{thm:intro:moderate} to rule out the system $\myvec{x_1} - \myvec{x_2} = 0$. It is not hard to see that this particular system is temperate but not moderate.}
In particular, one of the systems studied by Mimura and Tokushige does not meet these requirements (see \mysecref{sec:examples} for a detailed discussion).

We do not know if every irreducible linear system of type (RC) is moderate and/or temperate, but we have the following partial result.
We say that a balanced linear equation satisfied by $(\myvec{x_1},\ldots,\myvec{x_k}) \in S^k$ \emph{preserves the column equivalence classes of \mytag{}} if appending that equation to the system \mytag{} preserves the column equivalence classes.
We prove the following:

\begin{AlphTheorem}
	\label{thm:intro:rank}
	Let \mytag{}, $A$, $m$, $k$ and $\ell$ be as in \cref{sit:system}.
	Then there exist constants $\beta,\gamma > 0$ with $\gamma < q$ such that every subset $S \subseteq \F_q^n$ of size at least $\beta \cdot \gamma^n$ has a solution $(\myvec{x_1},\ldots,\myvec{x_k}) \in S^k$ of \mytag{} with the following properties:
	\begin{enumerate}[label=\textup(\roman*\textup)]
		\item\label{itm:intro:rnk:preserve} every balanced linear equation satisfied by $(\myvec{x_1},\ldots,\myvec{x_k})$ preserves the column equivalence classes of \mytag{};
		\item\label{itm:intro:rnk:dim-aff} $\dim(\aff(\myvec{x_1},\ldots,\myvec{x_k})) \geq \min(k - \ell , k - 2)$.
	\end{enumerate}
\end{AlphTheorem}

\noindent
\Cref{thm:intro:rank} improves upon \cref{thm:intro:Sauermann-rank} whenever $2 \leq \ell < 2m$; see \cref{rmk:rank-improvement}.

\medskip
Finally, we turn to an application of our techniques and results.
In characteristic $0$, results like Bourgain's theorem \cite{Bourgain-AP} (see also \cite[Chapter 12]{Tao-Vu}) show that it is substantially easier to find long arithmetic progressions in sum sets than in general sets.
Using the techniques from this paper, we establish a similar result in vector spaces over $\F_q$.

Given sets $S_1,\ldots,S_l \subseteq \F_q^n$, we define the \emph{affinely independent restricted sum set} (or \emph{AIR-sumset}) as follows:
\[ S_1 \AIRplus \cdots \AIRplus S_l := \{\myvec{x_1} + \cdots + \myvec{x_l} \, \mid \, \myvec{x_1} \in S_1,\ldots,\myvec{x_l} \in S_l\ \text{affinely independent}\}. \]
Further, if \mytag{} is linear system which is not necessarily balanced, then we say that a solution $(\myvec{x_1},\ldots,\myvec{x_k}) \in (\F_q^n)^k$ is \emph{linearly generic} if every linear equation (over $\F_q$) satisfied by $(\myvec{x_1},\ldots,\myvec{x_k})$ is a linear combination of the equations in \mytag{}.
By comparison, the solutions which we call \emph{generic} throughout this paper (see \mycref{def:non-degenerate}{itm:sol:generic}) only satisfy this property for \emph{balanced} linear equations (so by `generic' we will always mean `affinely generic').

\begin{AlphCorollary}
	\label{cor:intro:application-general}
	Let $\F_q$ be a finite field, let \mytag{} be a \textup(not necessarily balanced\textup) linear system over $\F_q$, and let $c_1,\ldots,c_l \in \F_q \setminus \{0\}$ with $c_1 + \cdots + c_l = 0$.
	Then there are constants $\beta,\gamma \geq 1$ with $\gamma < q$ such that, for every subset $S \subseteq \F_q^n$ of size at least $\beta \cdot \gamma^n$, the set \smash{$(c_1\cdot S \AIRplus \cdots \AIRplus c_l\cdot S) \cup \{0\}$} contains a linearly generic solution of \mytag{}.
\end{AlphCorollary}

Note that \cref{cor:intro:application-general} does not impose any restriction on the linear system \mytag{}; that is, the coefficient matrix $A \in \F_q^{m \times k}$ can be arbitrary.
This is a significant difference with our main results and Sauermann's result (\cref{thm:intro:Sauermann-invertible} above), which only work for very specific classes of linear systems.

In \cref{cor:intro:application-general}, we only need to append $0$ to the AIR-sumset when one of the single-variable equations $\myvec{x_j} = 0$ ($j \in [k]$) can be written as a linear combination of the equations in the linear system \mytag{}.
If this is not the case, then a linearly generic solution $(\myvec{x_1},\ldots,\myvec{x_k})$ will satisfy $\myvec{x_j} \neq 0$ for all $j \in [k]$, so it is not necessary to append $0$ to the AIR-sumset.

By letting \mytag{} be the system that encodes a $k$-term arithmetic progression, \cref{cor:intro:application-general} contains the following special case:

\begin{AlphCorollary}
	\label{cor:intro:application-AP}
	Let $p$ be prime, and let $3 \leq k \leq p$.
	Then, for every subset $S \subseteq \F_p^n$ of size at least $p^{1 + (1 - \frac{1}{k})n}$, the set $(S - S) \setminus \{0\}$ contains a non-trivial $k$-term arithmetic progression.
\end{AlphCorollary}

We note that this special case can be proved without using the slice rank method, using only a simple counting argument (see \mysecref{sec:examples} for details).

\subsection{Overview of the main ideas and organization of this paper}
\paragraph{Main ideas.}
There are two new techniques in this paper.

First, the majority of our results depend on a `replacement trick'.
This trick works roughly as follows.
If the $j_1$-th and $j_2$-th columns of $A$ are non-zero multiples of one another, and if we have a long enough list $\{(\myvec{x_1^{(i)}},\ldots,\myvec{x_k^{(i)}})\}_{i=1}^L$ of pairwise disjoint solutions to \mytag{}, then we use tricoloured sum-free sets to recombine these solutions to obtain new solutions of \mytag{}.
This is done by taking one of the solutions from this list, say $(\myvec{x_1^{(i)}},\ldots,\myvec{x_k^{(i)}})$, and replacing $\myvec{x_{j_1}^{(i)}}$ and $\myvec{x_{j_2}^{(i)}}$ by (respectively) $\myvec{x_{j_1}^{(i')}}$ and $\myvec{x_{j_2}^{(i'')}}$, for some $i',i'' \neq i$.
We show in \cref{cor:single-replacement} that there exists $i \in [L]$ which admits one such replacement (the `single replacement trick'), and in \cref{cor:multiple-replacement} that there exists $i \in [L]$ which admits many replacements (the `multiple replacement trick').

The second main ingredient in our proofs is \cref{lem:matrix-pigeonhole}, which shows that, for every subset $S \subseteq \F_q^n$ of size at least $q^{1 + (1 - \frac{1}{k})n}$, the difference set $S - S$ contains linearly generic solutions to every linear system in $k$ variables (even systems with $k < 2m + 1$).
The proof relies only on a simple counting argument, using the pigeonhole principle.

We point out that this paper does not make use of the full strength of \cref{thm:intro:Tao}, as we only use the slice rank method for $3$-tensors. Indeed, the replacement trick relies on tricoloured sum-free sets, and \cref{lem:matrix-pigeonhole} does not rely on the slice rank method at all.

\paragraph{The constants.}
\MyCref{thm:intro:moderate}{itm:intro:mod:nz}, \mycref{thm:intro:temperate}{itm:intro:tmp:nz}, and \cref{thm:intro:rank} rely only on the replacement trick.
Hence, the base of the exponent in the upper bounds from these theorems\hair\footnote{By `the base of the exponent in the upper bound', we mean the constant $\gamma < q$ in the upper bound $\beta \cdot \gamma^n$.} is equal to $\Gamma_q$, the constant from the bound on tricoloured sum-free sets (see \cref{thm:tricoloured}).

\MyCref{thm:intro:moderate}{itm:intro:mod:z}, \mycref{thm:intro:temperate}{itm:intro:tmp:z}, and \cref{cor:intro:application-general} rely on a combination of the replacement trick and \cref{lem:matrix-pigeonhole}. Hence, the base of the exponent in the upper bounds from these theorems is the maximum of $\Gamma_q$ and $q^{\frac{k-1}{k}}$.

\Cref{cor:intro:application-AP} relies solely on \cref{lem:matrix-pigeonhole}.
The base of the exponent in the upper bound is $p^{1 - \frac{1}{k}}$.

\paragraph{Organization of the paper.}
This paper consists of three parts.

First, in \mysecref{sec:moderate-prelims}--\ref{sec:mod:z}, we focus on moderate systems.
In \mysecref{sec:moderate-prelims}, we discuss the generalities of moderate systems, and we show that we may restrict our attention to irreducible systems.
In \mysecref{sec:mod:nz}, we establish the `single replacement trick', and use it to prove \mycref{thm:intro:moderate}{itm:intro:mod:nz}.
In \mysecref{sec:mod:z}, we establish the other main technique of this paper (\cref{lem:matrix-pigeonhole}), and combine it with the replacement trick to prove \mycref{thm:intro:moderate}{itm:intro:mod:z}.

Second, in \mysecref{sec:temperate-prelims}--\ref{sec:temperate-proofs}, we focus on temperate systems.
In \mysecref{sec:temperate-prelims}, we discuss the generalities of temperate systems.
Here we show how the problem of finding solutions of high rank is related to the problem of finding a generic solution, and we show that we may once again restrict our attention to irreducible systems.
In \mysecref{sec:temperate-proofs}, we establish the `multiple replacement trick', and use it to prove \cref{thm:intro:temperate} and \cref{thm:intro:rank}.

Finally, in \mysecref{sec:examples}, we discuss several examples and applications.
Here we prove \cref{cor:intro:application-general} and \cref{cor:intro:application-AP}, and we recover most of the results from \cite{Mimura-Tokushige-star,Mimura-Tokushige-shape,Mimura-Tokushige-II} as special cases of our results.
Furthermore, we show that the system conjectured to be moderate in \cite{Mimura-Tokushige-II} is indeed moderate.

\section{Preliminaries on moderate systems}
\label{sec:moderate-prelims}
In this paper, we study linear systems of the form
\[ \begin{cases}
	a_{11}\myvec{x_1} + \cdots + a_{1k}\myvec{x_k} = 0,\\[1ex]
	\hspace*{14.75mm}\vdots\\[1ex]
	a_{m1}\myvec{x_1} + \cdots + a_{mk}\myvec{x_k} = 0;
\end{cases}\tag*{\mytag} \]
with coefficient matrix $A = (a_{ij}) \in \F_q^{m\times k}$ and variables $\myvec{x_1},\ldots,\myvec{x_k} \in \F_q^n$.

Following standard usage, we say that two linear systems \mytag{} and \mytagprime{} are \emph{equivalent} if each equation in \mytag{} is a linear combination of the equations in \mytagprime{} and vice versa.
Furthermore, we say that a variable $\myvec{x_i}$ is \emph{used} by the linear system \mytag{} if it occurs with non-zero coefficient in at least one equation.
\begin{definition}
	\label{def:systems}
	The linear system \mytag{} is said to be:
	\begin{enumerate}[label=(\alph*)]
		\item \emph{non-degenerate} if the rows of $A$ are linearly independent and every variable is used (equivalently: $A$ has rank $m$ and $A$ has no zero columns);
		
		\item \emph{reducible} if it is equivalent to a linear system \mytagprime{} with the property that the variables $\myvec{x_1},\ldots,\myvec{x_k}$ can be partitioned into two or more classes in such a way that every equation in \mytagprime{} only uses variables from one partition class.
		If this is not the case, then \mytag{} is said to be \emph{irreducible}.
	\end{enumerate}
\end{definition}

Passing to an equivalent system or deleting columns with only zeroes does not change the problem of finding a \mytag{}-shape, so we may assume without loss of generality that \mytag{} is non-degenerate.
The following proposition shows that we can also restrict our attention to irreducible systems.

\begin{proposition}
	\label{prop:moderate-reduce}
	Suppose that \mytag{} is equivalent to a linear system \mytagprime{} whose coefficient matrix can be written as
	\[ \begin{pmatrix} A_1 & 0 \\ 0 & A_2 \end{pmatrix} \]
	for some $A_1 \in \F_q^{m_1 \times k_1}$ and $A_2 \in \F_q^{m_2 \times k_2}$ with $m_1,m_2,k_1,k_2 \neq 0$.
	Then \mytag{} is moderate if and only if the systems given by $A_1$ and $A_2$ are moderate.
\end{proposition}
\begin{proof}
	If \mytagprime{} is moderate with constants $\beta,\gamma>0$, where $\gamma<q$, then the same holds for the systems given by $A_1$ and $A_2$ as every \mytagprime{}-shape $(\myvec{x_1},\ldots,\myvec{x_{k_1+k_2}})$ yields an $A_1$-shape $(\myvec{x_1},\ldots,\myvec{x_{k_1}})$ and an $A_2$-shape $(\myvec{x_{k_1+1}},\ldots,\myvec{x_{k_1+k_2}})$.
	
	Conversely, suppose that for $i=1,2$, the system given by $A_i$ is moderate, with constants $\beta_i,\gamma_i > 0$, where $\gamma_i < q$.
	Let $S \subseteq \F_q^n$ be a set of size at least $\max(\beta_1\gamma_1^n,k_1+\beta_2\gamma_2^n)$.
	Since $|S| \geq \beta_1\gamma_1^n$, we may choose an $A_1$-shape $(\myvec{x_1},\ldots,\myvec{x_{k_1}})$ in $S$.
	Since $|S \setminus \{\myvec{x_1},\ldots,\myvec{x_{k_1}}\}| \geq \beta_2\gamma_2^n$, we may choose an $A_2$-shape $(\myvec{y_1},\ldots,\myvec{y_{k_2}})$ in $S \setminus \{\myvec{x_1},\ldots,\myvec{x_{k_1}}\}$. Then $(\myvec{x_1},\ldots,\myvec{x_{k_1}},\myvec{y_1},\ldots, \myvec{y_{k_2}})$ is a \mytagprime{}-shape. 
	Since $\max(\beta_1\gamma_1^n,k_1+\beta_2\gamma_2^n) \in \bigOh(\max(\gamma_1,\gamma_2)^n)$, this shows that \mytagprime{}, and therefore \mytag{}, is moderate.
\end{proof}

Therefore we may restrict our attention to irreducible systems, as stipulated in \cref{sit:system}.

The following proposition will be useful later on.

\begin{proposition}
	\label{prop:two-classes}
	Let \mytag{} be a linear system given by the matrix $A = (a_{ij}) \in \F_q^{m \times k}$.
	If \mytag{} is non-degenerate and irreducible, and if $m \geq 2$, then every non-zero linear equation implied by \mytag{} uses at least two column equivalence classes, and $\ell \geq m + 1$.
\end{proposition}
\begin{proof}
	Let $\ell$ be the number of column equivalence classes, and note that $m = \rank(A) \leq \ell$ (recall that the columns with indices in the same column equivalence class are scalar multiples of each other).
	Suppose for the sake of contradiction that some linear combination of the rows of \mytag{} uses exactly one column equivalence class.
	By passing to an equivalent system and permuting the columns, we may assume without loss of generality that the first row of \mytag{} only uses the column equivalence class $C = \{1,\ldots,|C|\} \subseteq [k]$.
	Since the columns indexed by $C$ are non-zero multiples of one another, we have $a_{1j} \neq 0$ for all $j \in C$.
	
	By Gaussian elimination, we may pass to an equivalent system \mytagprime{}, given by the matrix $A' = (a_{ij}') \in \F_q^{m \times k}$, such that $a_{i1} = 0$ for all $i > 1$.
	Since elementary row operations preserve the column equivalence classes, we have $a_{ij} = 0$ for all $(i,j) \in \{2,\ldots,m\} \times C$.
	It follows that every row in \mytagprime{} uses variables from either $C$ or $[k] \setminus C$, but not both.
	Since $\ell \geq m \geq 2$, we have $|C|,|[k] \setminus C| \neq 0$, so it follows that \mytag{} is reducible.
	This is a contradiction, so we conclude that every (non-zero) equation implied by \mytag{} uses at least two column equivalence classes.
	
	To prove that $\ell \geq m + 1$, let $A'$ be the matrix obtained by deleting from $A$ the columns in one column equivalence class. By the above, every non-zero element of the row space of $A'$ uses at least one of the remaining $\ell-1$ column equivalence classes. It follows that $\rank(A')=m$, so $\ell-1\geq m$.
\end{proof}

\begin{remark}
	\label{rem:AboutSit1.7}
	In Theorems~\ref{thm:intro:moderate} -- \ref{thm:intro:rank} we are in \cref{sit:system}; that is, our system is of type (RC) and is irreducible and non-degenerate. For context, we mention two facts about this situation that are not needed in the proofs, but may be helpful nonetheless.
	\begin{enumerate}[label=\textup{(\roman*)}]
		\item \label{rem:AboutSit1.7Part1}
		If $m\geq 2$, then $k\geq 2m+1$. This follows directly from \cref{prop:two-classes} and the fact that, in a type (RC) linear system, every column equivalence class except at most one must have size at least two. 
		
		\item \label{rem:AboutSit1.7Part2} If $\ell=m+1$, then either every column class sums to zero, or none of the column classes sums to zero. Indeed, after row operations and permuting columns we may assume that $A=\begin{bmatrix}I&B\end{bmatrix}$, where every column of $B$ is a scalar multiple of one of the vectors in $\{\myvec{e}_1,\ldots, \myvec{e}_m, \myvec{b}\}$, and $\myvec{b}$ has no zero entries. So for every $i\in [m]$, the union of the column equivalence classes of $\myvec{e}_i$ and $\myvec{b}$ sums to zero.
	\end{enumerate}
\end{remark}

\section{Proof of \myautorefstar{thm:intro:moderate}{itm:intro:mod:nz}}
\label{sec:mod:nz}

In this section, we develop the first main technique (the `single replacement trick', see \cref{cor:single-replacement}) and use it to prove \mycref{thm:intro:moderate}{itm:intro:mod:nz}.

\begin{definition}
	Let $G$ be an abelian group.
	A sequence $\{(x_i,y_i,z_i)\}_{i=1}^L$ in $G^3$ is called a \emph{tricoloured sum-free set in $G$} if for all $i,i',i'' \in [L]$ one has $x_i + y_{i'} + z_{i''} = 0$ if and only if $i = i' = i''$.
\end{definition}
Note that the definition implies $|\{x_1,\ldots,x_L\}| = |\{y_1,\ldots,y_L\}| = |\{z_1,\ldots,z_L\}| = L$; that is, in a tricoloured sum-free set there can be no repetitions in each of the coordinates (separately).

For all positive integers $t\geq 2$, define
\[ J(t) := \frac{1}{t}\min_{0<x<1} \frac{1+x+\cdots+x^{t-1}}{x^\frac{t-1}{3}}. \]
It follows from \cite[Prop.{} 4.12]{Blasiak-et-al} that $J(t)$ is decreasing in $t$.
Hence, for all $t \geq 2$ we have $J(t) \leq J(2) = 3 \cdot 2^{-5/3} < 0.945$.
For a prime power $q$, define $\Gamma_q := qJ(q) < 0.945 q$.\hair\footnote{Alternatively, for a prime power $q = p^s$, one could define $\Gamma_q := (pJ(p))^s < 0.945^s \: q$. This gives a slightly better bound, and \cref{thm:tricoloured} remains true as stated, because $\F_q^n \cong \F_p^{sn}$ as groups.}
By a routine application of the slice rank method, one can prove the following bound on the size of tricoloured sum-free sets.
\begin{theorem}[{\cite{Blasiak-et-al}}]
	\label{thm:tricoloured}
	Let $q$ be a prime power, and let $\{(\myvec{x_i},\myvec{y_i},\myvec{z_i})\}_{i=1}^L$ be a tricoloured sum-free set in $\F_q^n$. Then $L < (\Gamma_q)^n$.
\end{theorem}

To recover \cref{thm:tricoloured} from known results, one has to proceed in three steps.
First, the bound $L \leq 3\cdot (\Gamma_q)^n$ follows from \cite{Blasiak-et-al}.
Second, to get rid of the additional factor $3$, use the `power trick' (a tricoloured sum-free set of size $L$ in $\F_q^n$ gives one of size $L^N$ in $\F_q^{nN}$ for all $N$).
Finally, to get a strict inequality, prove that $(\Gamma_q)^n$ is never an integer (see for instance \cite[Remark~5.11]{Dobben-dissertation}).
Alternatively, the results in this paper can be recovered by passing to a marginally higher constant $\Gamma_q + \varepsilon$ instead of $\Gamma_q$, because we have $L \leq (\Gamma_q)^n < (\Gamma_q + \varepsilon)^n$ for all $\varepsilon > 0$.

To prove the `single replacement trick', we start with the following lemma.

\begin{lemma}
	\label{lem:pairs}
	Let $q$ be a prime power, and let $\Gamma_q$ be as in \cref{thm:tricoloured}.
	Let $\alpha,\beta \in \F_q \setminus \{0\}$, let $\myvec{x_1} , \ldots , \myvec{x_L} \in \F_q^n$ be distinct, and let $\myvec{y_1} , \ldots , \myvec{y_L} \in \F_q^n$ be distinct.
	If $L \geq (\Gamma_q)^n$, then there exist $i,i',i'' \in [L]$ with $i \neq i',i''$ and $\alpha \myvec{x_{i}} + \beta \myvec{y_{i}} = \alpha \myvec{x_{i'}} + \beta \myvec{y_{i''}}$.
\end{lemma}
\begin{proof}
	For all $i\in [L]$, define $\myvec{z_i} = \alpha \myvec{x_i} + \beta \myvec{y_i}$.
	Each triple in the sequence $\{(\alpha \myvec{x_i}, \beta \myvec{y_i}, -\myvec{z_i})\}_{i=1}^L$ sums to zero, but we have $L \geq (\Gamma_q)^n$, so it follows from \cref{thm:tricoloured} that this sequence is not a tricoloured sum-free set.
	Therefore we may choose $i,i',i'' \in [L]$, not all equal, such that $\alpha \myvec{x_i} + \beta \myvec{y_i} = \myvec{z_i} = \alpha \myvec{x_{i'}} + \beta \myvec{y_{i''}}$.
	
	Suppose that $i'' = i$.
	Then we have $\alpha \myvec{x_i} = \alpha \myvec{x_{i'}}$, hence $\myvec{x_i} = \myvec{x_{i'}}$ (because $\alpha \neq 0$), and therefore $i = i'$ (because $\myvec{x_1},\ldots,\myvec{x_L}$ are distinct), contrary to our assumption that $i$, $i'$ and $i''$ are not all equal.
	This is a contradiction, so we must have $i'' \neq i$.
	An analogous argument shows that $i' \neq i$.
\end{proof}
\begin{remark}
	In \cref{lem:pairs}, we do not require that $i' \neq i''$.
	The case that $i' = i''$ corresponds to the case that $\myvec{z_1},\ldots,\myvec{z_L}$ are not all distinct.
	This does not matter for the rest of the proof.
\end{remark}

\begin{definition}
	We say that two solutions $\myvecvec{x} = (\myvec{x_1},\ldots,\myvec{x_k})$ and $\myvecvec{y} = (\myvec{y_1},\ldots, \myvec{y_k})$ to \mytag{} are \emph{disjoint} if $\{\myvec{x_1},\ldots,\myvec{x_k}\} \cap \{\myvec{y_1},\ldots,\myvec{y_k}\} = \varnothing$. Note that we do not require the $\myvec{x_j}$ (resp.{} the $\myvec{y_j}$) to be pairwise distinct.
\end{definition}
\begin{corollary}[`Single replacement trick']
	\label{cor:single-replacement}
	Let $\{(\myvec{x_1^{(i)}},\ldots,\myvec{x_k^{(i)}})\}_{i=1}^L$ be a list of pairwise disjoint solutions of \mytag{}, and suppose that $j_1$ and $j_2$ are distinct indices in the same column equivalence class.
	If $L \geq (\Gamma_q)^n$, then there exist $i,i',i'' \in [L]$ with $i \neq i',i''$ such that the $k$-tuple $(\myvec{y_1},\ldots,\myvec{y_k}) \in (\F_q^n)^k$ given by
	\[ \myvec{y_j} = \begin{cases}
		\myvec{x_j^{(i)}},&\quad\text{if $j \neq j_1,j_2$};\\[1ex]
		\myvec{x_j^{(i')}},&\quad\text{if $j = j_1$};\\[1ex]
		\myvec{x_j^{(i'')}},&\quad\text{if $j = j_2$};
	\end{cases} \]
	is also a solution of \mytag{}.
\end{corollary}
\begin{proof}
	Since the $j_1$-th and $j_2$-th column of \mytag{} are multiples of one another, we may choose a vector $\myvec{v} \in \F_q^m$ and constants $\alpha,\beta \neq 0$ such that the $j_1$-th column is equal to $\alpha \myvec{v}$ and the $j_2$-th column is equal to $\beta \myvec{v}$.
	
	By assumption, the vectors $\myvec{x_{j_1}^{(1)}},\ldots,\myvec{x_{j_1}^{(L)}}$ are distinct, and likewise the vectors $\myvec{x_{j_2}^{(1)}},\ldots,\myvec{x_{j_2}^{(L)}}$ are distinct, so it follows from \cref{lem:pairs} that there exist $i,i',i'' \in [L]$ with $i \neq i',i''$ and $\alpha \myvec{x_{j_1}^{(i)}} + \beta \myvec{x_{j_2}^{(i)}} = \alpha \myvec{x_{j_1}^{(i')}} + \beta \myvec{x_{j_2}^{(i'')}}$.
	Hence, the total contribution of $\myvec{x_{j_1}^{(i)}}$ and $\myvec{x_{j_2}^{(i)}}$ to the equations of \mytag{} is the same as the contribution of $\myvec{x_{j_1}^{(i')}}$ and $\myvec{x_{j_2}^{(i'')}}$.
	Since $(\myvec{x_1^{(i)}},\ldots,\myvec{x_k^{(i)}})$ is a solution of \mytag{}, so is $(\myvec{y_1},\ldots,\myvec{y_k})$.
\end{proof}

We now prove the first main result of this paper, using the replacement trick from the preceding corollary.

\begin{proof}[{Proof of \protect\mycref{thm:intro:moderate}{itm:intro:mod:nz}}]
	Let \mytag{}, $A$, $m$, $k$ and $\ell$ be as in \cref{sit:system}, and suppose that \mytag{} satisfies property~\ref{itm:intro:mod:nz} from \cref{thm:intro:moderate} (none of the column equivalence classes of size $2$ sums to zero).
	Furthermore, let $\Gamma_q$ be the constant from \cref{thm:tricoloured}.
	
	We prove by induction on $\lambda$ that, for every $\lambda \in [k]$, there is a constant $\beta_\lambda \geq 1$ such that every subset $S \subseteq \F_q^n$ of size at least $\beta_\lambda \cdot (\Gamma_q)^n$ contains a solution $(\myvec{x_1},\ldots,\myvec{x_k}) \in S^k$ of \mytag{} with at least $\lambda$ different vectors; that is, $|\{\myvec{x_1},\ldots,\myvec{x_k}\}| \geq \lambda$.
	Setting $\lambda = k$ then proves the theorem.
	
	For $\lambda = 1$, the claim is trivially true with $\beta_1 = 1$, since $(\myvec{x},\ldots,\myvec{x})$ is a solution of \mytag{} for every $\myvec{x} \in \F_q^n$.
		
	For the induction step, suppose that $\lambda_0 \in [k - 1]$ is given such that the statement is true for $\lambda = \lambda_0$.
	Define $\beta_{\lambda_0 + 1} := \beta_{\lambda_0} + P(k,\lambda_0)\cdot k$, where $P(k,\lambda_0)$ denotes the number of partitions of a $k$-element set into $\lambda_0$ parts.
	
	Let $S \subseteq \F_q^n$ be a set of size at least $\beta_{\lambda_0 + 1} \cdot (\Gamma_q)^n = \beta_{\lambda_0} \cdot (\Gamma_q)^n + P(k,\lambda_0) \cdot (\Gamma_q)^n \cdot k$.
	Create a list of disjoint solutions $\{(\myvec{x_1^{(i)}},\ldots,\myvec{x_k^{(i)}})\}_{i=1}^{L_0}$ of \mytag{} in $S$, each with at least $\lambda_0$ different vectors, by repeatedly finding such a solution in $S$ and removing it from $S$.
	By the induction hypothesis, we can find a new solution as long as the remaining set has size at least $\beta_{\lambda_0} \cdot (\Gamma_q)^n$, and in each step we remove at most $k$ vectors from $S$, so we find a list of length $L_0 \geq P(k,\lambda_0) \cdot (\Gamma_q)^n$.
	
	If one of the solutions in the list has strictly more than $\lambda_0$ different vectors, then we are done.
	So we may assume that every solution in the list has exactly $\lambda_0$ different vectors.
	
	We sort the entries in the list according to their partition pattern.
	We say that a solution $(\myvec{x_1^{(i)}},\ldots,\myvec{x_k^{(i)}})$ is \emph{compatible} with a partition $[k] = J_1 \cup \cdots \cup J_{\lambda_0}$ if for all $j_1,j_2\in [k]$ we have: $\myvec{x_{j_1}^{(i)}} = \myvec{x_{j_2}^{(i)}}$ if and only if $j_1$ and $j_2$ belong to the same partition class.
	Evidently every solution is compatible with exactly one partition.
	By the pigeonhole principle, we may choose a partition $[k] = J_1 \cup \cdots \cup J_{\lambda_0}$ that occurs at least $(\Gamma_q)^n$ times in our list of solutions.
	Thus, we obtain a list $\{(\myvec{y_1^{(i)}},\ldots,\myvec{y_k^{(i)}})\}_{i=1}^{L_1}$ of solutions of the same partition type, where $L_1 \geq (\Gamma_q)^n$.
	
	Now we have two competing partitions of $[k]$, given by the column equivalence classes and the (now fixed) partition type $[k] = J_1 \cup \cdots \cup J_{\lambda_0}$.
	For $j_1,j_2 \in [k]$, we write $j_1 \coleq j_2$ if $j_1$ and $j_2$ are in the same column equivalence class, and $j_1 \parteq j_2$ if $j_1$ and $j_2$ belong to the same class in the partition $[k] = J_1 \cup \cdots \cup J_{\lambda_0}$ (i.e.{} if $\myvec{y_{j_1}^{(i)}} = \myvec{y_{j_2}^{(i)}}$ for all $i \in [L_1]$).
	
	Since $\lambda_0 < k$, we may choose distinct $j_0,j_1 \in [k]$ with $j_0 \parteq j_1$.
	Furthermore, since \mytag{} is a type (RC) linear system (see \mycref{def:equivalence}{itm:equiv:type-DC}), it has at most one column equivalence class of size $1$, so we may assume without loss of generality that $j_1$ belongs to a column equivalence class of size $2$ or more.
	We distinguish two cases, depending on which of the column equivalence classes $j_0$ and $j_1$ belong to.
	\begin{itemize}
		\item \textbf{Case 1:} $j_0 \ncoleq j_1$ or $j_0$ and $j_1$ belong to the same column equivalence class of size at least $3$. In this case, we may choose $j_2 \neq j_0,j_1$ such that $j_1 \coleq j_2$.
		By \cref{cor:single-replacement}, there is a solution $(\myvec{z_1},\ldots,\myvec{z_k})$ of \mytag{} of the form
		\[ \myvec{z_j} = \begin{cases}
			\myvec{y_j^{(i)}},&\quad\text{if $j \neq j_1,j_2$};\\[1ex]
			\myvec{y_j^{(i')}},&\quad\text{if $j = j_1$};\\[1ex]
			\myvec{y_j^{(i'')}},&\quad\text{if $j = j_2$};
		\end{cases} \]
		for some $i,i',i'' \in [L_1]$ with $i \neq i',i''$.
		In other words, $(\myvec{z_1},\ldots,\myvec{z_k})$ is obtained by taking the solution $(\myvec{y_1^{(i)}},\ldots,\myvec{y_k^{(i)}})$ and replacing two entries.
		
		We prove that $|\{\myvec{z_1},\ldots,\myvec{z_k}\}| \geq \lambda_0 + 1$.
		First, note that $\{\myvec{z_{j_1}},\myvec{z_{j_2}}\} \cap \{\myvec{z_j} \mid j \neq j_1,j_2\} = \varnothing$, since the solutions in the list were disjoint.
		Now we distinguish two cases.
		\begin{itemize}
			\item If $j_1 \parteq j_2$, then the removal of the $j_1$-th and $j_2$-th vectors from $(\myvec{y_1^{(i)}},\ldots,\myvec{y_k^{(i)}})$ does not change the number of different vectors, since $\myvec{y_{j_0}^{(i)}} = \myvec{y_{j_1}^{(i)}} = \myvec{y_{j_2}^{(i)}}$.
			We replace them by two vectors $\myvec{z_{j_1}},\myvec{z_{j_2}}$ which are distinct from the other vectors in the solution (but possibly $\myvec{z_{j_1}} = \myvec{z_{j_2}}$), so the number of different vectors increases by at least $1$.
			
			\item If $j_1 \nparteq j_2$, then the removal of $j_1$-th and $j_2$-th vectors from $(\myvec{y_1^{(i)}},\ldots,\myvec{y_k^{(i)}})$ decreases the number of different vectors by at most $1$, because $\myvec{y_{j_0}^{(i)}} = \myvec{y_{j_1}^{(i)}}$.
			In this case we are guaranteed to have $\myvec{z_{j_1}} \neq \myvec{z_{j_2}}$: different solutions in the list are disjoint, but even within the same solution the $j_1$-th and $j_2$-th entry are always different (because $j_1 \nparteq j_2$).
			Thus, adding $\myvec{z_{j_1}}$ and $\myvec{z_{j_2}}$ to the solution increases the number of different vectors by $2$.
			The net effect is an increase of at least $1$.
		\end{itemize}
		This proves our claim that $|\{\myvec{z_1},\ldots,\myvec{z_k}\}| \geq \lambda_0 + 1$.
		
		\item \textbf{Case 2:} $j_0$ and $j_1$ belong to the same column equivalence class of size $2$.
		Then, by assumption~\ref{itm:intro:mod:nz} from the theorem statement, the $j_0$-th and $j_1$-th columns of \mytag{} do not sum to zero.
		
		By \cref{cor:single-replacement}, there is a solution $(\myvec{z_1},\ldots,\myvec{z_k})$ of \mytag{} of the form
		\[ \myvec{z_j} = \begin{cases}
			\myvec{y_j^{(i)}},&\quad\text{if $j \neq j_0,j_1$};\\[1ex]
			\myvec{y_j^{(i')}},&\quad\text{if $j = j_0$};\\[1ex]
			\myvec{y_j^{(i'')}},&\quad\text{if $j = j_1$};
		\end{cases} \]
		for some $i,i',i'' \in [L_1]$ with $i \neq i',i''$.
		
		Suppose for the sake of contradiction that $\myvec{z_{j_0}} = \myvec{z_{j_1}}$; that is, $\myvec{y_{j_0}^{(i')}} = \myvec{y_{j_1}^{(i'')}}$.
		Since the $j_0$-th and $j_1$-th columns of \mytag{} do not sum to zero, and since $\myvec{y_{j_0}^{(i)}} = \myvec{y_{j_1}^{(i)}}$, the fact that both $(\myvec{y_1^{(i)}},\ldots,\myvec{y_k^{(i)}})$ and $(\myvec{z_1},\ldots,\myvec{z_k})$ are solutions of \mytag{} implies that $\myvec{y_{j_0}^{(i')}} = \myvec{y_{j_1}^{(i'')}} = \myvec{y_{j_0}^{(i)}} = \myvec{y_{j_1}^{(i)}}$.
		This is a contradiction, because $i \neq i',i''$, and different solutions of the list are disjoint.
		Therefore we must have $\myvec{z_{j_0}} \neq \myvec{z_{j_1}}$.
		
		The removal of $\myvec{y_{j_0}^{(i)}}$ and $\myvec{y_{j_1}^{(i)}}$ from the solution decreases the number of different vectors by at most $1$, since $\myvec{y_{j_0}^{(i)}} = \myvec{y_{j_1}^{(i)}}$.
		On the other hand, putting back $\myvec{z_{j_0}}$ and $\myvec{z_{j_1}}$ increases the number of different vectors by $2$, since we have $\myvec{z_{j_0}} \neq \myvec{z_{j_1}}$ and $\{\myvec{z_{j_1}},\myvec{z_{j_2}}\} \cap \{\myvec{z_j} \mid j \neq j_1,j_2\} = \varnothing$.
		The net effect is an increase of at least $1$, so we have $|\{\myvec{z_1},\ldots,\myvec{z_k}\}| \geq \lambda_0 + 1$.
		\qedhere
	\end{itemize}
\end{proof}

\section{Proof of \myautorefstar{thm:intro:moderate}{itm:intro:mod:z}}
\label{sec:mod:z}

In this section, we develop our second main technique and combine it with the techniques from the previous section to prove \mycref{thm:intro:moderate}{itm:intro:mod:z}.

Our second main technique is the following lemma, which shows that, for every subset $S \subseteq \F_q^n$ of size at least $q^{1 + (1 - \frac{1}{k})n}$, the difference set $S - S$ contains linearly generic solutions to every linear system in $k$ variables (including systems with $k < 2m + 1$).\hair\footnote{We stress that the added generality of omitting the assumption that $k \geq 2m + 1$ will be needed in applications of \cref{lem:matrix-pigeonhole} in the proofs of \mycref{thm:intro:moderate}{itm:intro:mod:z}, \mycref{thm:intro:temperate}{itm:intro:tmp:z} and \cref{thm:intro:rank}, because there we apply \cref{lem:matrix-pigeonhole} to a system which has one column from each column equivalence class from the original system \mytag{}.}
The proof uses a simple counting argument and does not rely on the slice rank method at all.

\begin{lemma}
	\label{lem:matrix-pigeonhole}
	Let $A=(a_{ij}) \in \F_q^{m \times k}$ be a non-zero matrix and let $S \subseteq \F_q^n$ have size at least $q^{1 + (1 - \frac{1}{k})n}$. Then there are $(\myvec{x_1},\ldots,\myvec{x_k}),(\myvec{y_1},\ldots,\myvec{y_k}) \in S^k$ such that, for all $\myvec{b} = (b_1,\ldots,b_k) \in \F_q^k$, one has $b_1\myvec{x_1} + \cdots + b_k\myvec{x_k} = b_1\myvec{y_1} + \cdots + b_k\myvec{y_k}$ if and only if $\myvec{b}$ is in the row space of $A$.
\end{lemma}
\begin{proof}
	By removing redundant rows, we may assume without loss of generality that $\rank A=m$. If $k=m$, then we can take $\myvec{x} = \myvec{y} \in S^k$ arbitrary. Hence, we may assume that $k\geq m+1$. By performing elementary row operations and permuting columns, we may assume without loss of generality that $A$ is of the form $[A'\ I_m]$ for some $A' \in \F_q^{m \times (k - m)}$.
	
	The matrix $A$ defines a function $f : (\F_q^n)^k \to (\F_q^n)^m$, where $[f(\myvec{x_1},\ldots,\myvec{x_k})]_i = a_{i1} \myvec{x_1} + \cdots + a_{ik} \myvec{x_k}$.
	By the pigeonhole principle, we may choose some $\myvecvec{z} = (\myvec{z_1},\ldots,\myvec{z_m}) \in (\F_q^n)^m$ such that the set $T := f^{-1}(\myvecvec{z}) \cap S^k$ has size $|T| \geq |S|^k/q^{mn} \geq q^k q^{(k-m-1)n}$.
	
	Let $\pi : (\F_q^n)^k \to (\F_q^n)^{k-m}$ be the projection onto the first $k-m$ coordinates, let $g : T \to (\F_q^n)^{k-m}$ be the restriction of $\pi$ to $T$, and let $T' := g[T]$.
	Since $A$ is of the form $[A'\ I_m]$, it is easy to see that for every $(\myvec{x_1},\ldots,\myvec{x_{k-m}}) \in (\F_q^n)^{k-m}$ there is exactly one possible choice of $(\myvec{x_{k-m+1}},\ldots,\myvec{x_k}) \in (\F_q^n)^m$ such that $f(\myvec{x_1},\ldots,\myvec{x_k}) = \myvecvec{z}$.
	Therefore $g$ is injective, and it follows that $|T'| = |T|$.
	
	Let $D = \{ (\myvec{w_1},\ldots,\myvec{w_{k-m}}) \in (\F_q^n)^{k-m} \mid \myvec{w_1},\ldots,\myvec{w_{k-m}} \ \text{are linearly dependent}\}$.
	Then $|D| < q^{k-m} q^{(k-m-1)n}$ since there are fewer than $q^{k-m}$ possible linear relations.
	
	Choose some $\myvecvecprime{y} = (\myvec{y_1'},\ldots,\myvec{y_{k-m}'}) \in T'$.
	Since $|T' - \myvecvecprime{y}| = |T'| > |D|$, we have $(T - \myvecvecprime{y}) \setminus D \neq \varnothing$, so we may choose $(\myvec{x_1'},\ldots,\myvec{x_{k-m}'}) \in T'$ such that $\myvec{x_1'} - \myvec{y_1'} , \ldots , \myvec{x_{k-m}'} - \myvec{y_{k-m}'}$ are linearly independent.
	Let $(\myvec{x_1},\ldots,\myvec{x_k}),(\myvec{y_1},\ldots,\myvec{y_k}) \in T \subseteq S^k$ be the (unique) preimages of $(\myvec{x_1'},\ldots,\myvec{x_{k-m}'})$ and $(\myvec{y_1'},\ldots,\myvec{y_{k-m}'})$ under $g$.
	Note that $(\myvec{x_1},\ldots,\myvec{x_{k-m}}) = (\myvec{x_1'},\ldots,\myvec{x_{k-m}'})$ and $(\myvec{y_1},\ldots,\myvec{y_{k-m}}) = (\myvec{y_1'},\ldots,\myvec{y_{k-m}'})$, since $g$ is just a coordinate projection.
	
	We claim that $(\myvec{x_1},\ldots,\myvec{x_k})$ and $(\myvec{y_1},\ldots,\myvec{y_k})$ satisfy the required property.
	
	Since $f(\myvec{x_1},\ldots,\myvec{x_k}) = f(\myvec{y_1},\ldots,\myvec{y_k}) = \myvecvec{z}$, it is clear that $b_1\myvec{x_1} + \cdots + b_k\myvec{x_k} = b_1\myvec{y_1} + \cdots + b_k\myvec{y_k}$ whenever $(b_1,\ldots,b_k)$ is in the row space of $A$.
	
	Now let $\myvec{b} = (b_1,\ldots,b_k) \in \F_q^k$ be an arbitrary row vector such that $b_1\myvec{x_1} + \cdots + b_k\myvec{x_k} = b_1\myvec{y_1} + \cdots + b_k\myvec{y_k}$.
	Since $A$ is of the form $[A'\ I_m]$, we can add a linear combination of the rows of $A$ to $\myvec{b}$ to obtain a vector $\myvec{c} = (c_1,\ldots,c_k) \in \F_q^k$ with $c_{k-m+1} = \cdots = c_k = 0$.
	By linearity, we have $c_1\myvec{x_1} + \cdots + c_k\myvec{x_k} = c_1\myvec{y_1} + \cdots + c_k\myvec{y_k}$, or equivalently,
	\[ c_1(\myvec{x_1} - \myvec{y_1}) + \cdots + c_{k-m}(\myvec{x_{k-m}} - \myvec{y_{k-m}}) = 0. \]
	Since $\myvec{x_1} - \myvec{y_1},\ldots,\myvec{x_{k-m}} - \myvec{y_{k-m}}$ are linearly independent, it follows that $c_1 = \cdots = c_{k-m} = 0$, so we have $c_j = 0$ for all $j \in [k]$.
	This shows that $\myvec{b}$ is in the row space of $A$.
\end{proof}

We now come to the proof of \mycref{thm:intro:moderate}{itm:intro:mod:z}.
The proof is largely analogous to the proof of \mycref{thm:intro:moderate}{itm:intro:mod:nz} (see \mysecref{sec:mod:nz}), the main difference being that we now use \cref{lem:matrix-pigeonhole} to control column equivalence classes that sum to zero.

We prove the following slightly stronger theorem.

\begin{theorem}
	\label{thm:moderate:stronger}
	Let \mytag{}, $A$, $m$, $k$ and $\ell$ be as in \cref{sit:system}. Suppose that there is a partition $[k] = P_1 \cup \cdots \cup P_{2s}$ such that:
	\begin{enumerate}[label=\textup{(\roman*)},beginpenalty=100,midpenalty=999,endpenalty=100]
		\item\label{itm:mod:stronger:z} for all $r\in [s]$, the columns of $A$ indexed by $P_r\cup P_{s+r}$ sum to zero;
		\item\label{itm:mod:stronger:nz} if $(b_1,\ldots,b_k) \in \F_q^k \setminus \{0\}$ is a non-zero element in the row space of $A$, then one has $\sum_{j \in P_r} b_j \neq 0$ for at least two different values of $r \in [s]$.\hair\footnote{Note that we only look at $r \in \{1,\ldots,s\}$, and we ignore all $r \in \{s+1,\ldots,2s\}$. This is because it follows from \ref{itm:mod:stronger:z} that $\sum_{j \in P_r} b_j \neq 0$ if and only if $\sum_{j \in P_{s + r}} b_j \neq 0$. An equivalent statement is that $\sum_{j \in P_r} b_j \neq 0$ for at least four different values of $r \in [2s]$.}
		
		\item\label{itm:mod:stronger:eq} if $C$ is a column equivalence class of size $2$ that sums to zero, then there is some $r \in [s]$ such that $C = P_r \cup P_{s + r}$.
	\end{enumerate}
	Then \mytag{} is moderate.
\end{theorem}
Before we prove \cref{thm:moderate:stronger}, we first show how it implies \mycref{thm:intro:moderate}{itm:intro:mod:z}.
\begin{proof}[{Proof of \protect\mycref{thm:intro:moderate}{itm:intro:mod:z}, assuming \cref{thm:moderate:stronger}}]
	Let $C_1,\ldots,C_\ell \subseteq [k]$ be the column equivalence classes of $A$.
	We distinguish two cases:
	\begin{itemize}
		\item If $\ell = 1$, then we have $m = \rank(A) \leq \ell = 1$, so we are in the situation with a single equation. Since we assumed $k \geq 3$, there is no column equivalence class of size $2$, so it follows from \mycref{thm:intro:moderate}{itm:intro:mod:nz} that \mytag{} is moderate.
		
		\item Suppose that $\ell \geq 2$.
		Since $A$ is non-degenerate, every column of $A$ is non-zero.
		Hence, since the column equivalence classes of $A$ sum to zero, every column equivalence class has size at least $2$.
		For every $r \in [\ell]$, choose $j_r \in C_r$ arbitrary, and set $P_r := \{j_r\}$ and $P_{\ell + r} := C_r \setminus \{j_r\}$.
		
		We prove that the partition $[k] = P_1 \cup \cdots \cup P_{2\ell}$ satisfies the properties from \cref{thm:moderate:stronger}.
		Property~\ref{itm:mod:stronger:z} is met because each of the column equivalence classes sums to zero, and property~\ref{itm:mod:stronger:eq} is met by construction.
		To see that property~\ref{itm:mod:stronger:nz} is met, recall that \mytag{} is irreducible, so it follows from \cref{prop:two-classes} that every non-zero element of the row space of $A$ uses at least two different column equivalence classes.
		\qedhere
	\end{itemize}
\end{proof}

\begin{proof}[{Proof of \cref{thm:moderate:stronger}}]
	Let $\Gamma_q$ be the constant from \cref{thm:tricoloured}.
	We prove by induction on $\lambda$ that, for every $\lambda \in [k]$, there is a constant $\beta_\lambda \geq 1$ such that every subset $S \subseteq \F_q^n$ of size at least $\beta_\lambda \cdot (\max(\Gamma_q,q^{\frac{k-1}{k}}))^n$ contains a solution $(\myvec{x_1},\ldots,\myvec{x_k}) \in S^k$ of \mytag{} satisfying the following properties:
	\begin{enumerate}[label=(\alph*)]
		\item the solution contains at least $\lambda$ different vectors; that is, $|\{\myvec{x_1},\ldots,\myvec{x_k}\}| \geq \lambda$;
		\item\label{itm:pf:mod:z:nz} for every column equivalence class of size $2$ that sums to zero, the variables $\myvec{x_{j_1}},\myvec{x_{j_2}}$ corresponding to that class are distinct.
	\end{enumerate}
	Before proving the base case, we first show that the induction step from the proof of \mycref{thm:intro:moderate}{itm:intro:mod:nz} carries through unchanged.
	This time, part~\ref{itm:pf:mod:z:nz} of the induction hypothesis replaces the assumption~\ref{itm:intro:mod:nz} from \cref{thm:intro:moderate}.
	To see that property~\ref{itm:pf:mod:z:nz} is automatically maintained by the proof of \mycref{thm:intro:moderate}{itm:intro:mod:nz}, recall that the induction step consists of choosing a column equivalence class $C_t$ and replacing two variables from that class by other values, leaving the other classes unchanged.
	Since we started and ended with a solution of \mytag{}, the contribution of the variables $\{\myvec{x_j} \mid j \in C_t\}$ to \mytag{} must have remained the same.
	Property~\ref{itm:pf:mod:z:nz} is equivalent to saying that the contribution of $\{\myvec{x_j} \mid j \in C\}$ to \mytag{} is non-zero for every column equivalence class $C$ of size $2$ that sums to zero, so this property is automatically maintained by the proof of \mycref{thm:intro:moderate}{itm:intro:mod:nz}.
	
	\medskip
	It remains to prove the base case.
	Let $B = (b_{ir}) \in \F_q^{m \times s}$ be the matrix given by
	\[ b_{ir} \: := \: \sum_{j \in P_r} a_{ij} \: = \ -\!\!\!\!\!\sum_{j \in P_{s + r}} a_{ij}. \]
	Suppose that $S \subseteq \F_q^n$ has size at least $q \cdot (\max(\Gamma_q,q^{\frac{k-1}{k}}))^n$.
	It follows from \cref{lem:matrix-pigeonhole} that there are $(\myvec{z_1},\ldots,\myvec{z_s}),(\myvec{z_{s+1}},\ldots,\myvec{z_{2s}}) \in S^s$ such that, for all $(c_1,\ldots,c_s) \in \F_q^s$, one has $c_1 \myvec{z_1} + \cdots + c_s \myvec{z_s} = c_1 \myvec{z_{s+1}} + \cdots + c_s \myvec{z_{2s}}$ if and only if $(c_1,\ldots,c_s)$ is in the row space of $B$.
	By assumption~\ref{itm:mod:stronger:nz}, none of the standard unit vectors $\myvec{e_1},\ldots,\myvec{e_s} \in \F_q^s$ is in the row space of $B$, so it follows that $\myvec{z_r} \neq \myvec{z_{s+r}}$ for all $r\in [s]$ (since $\myvec{z_{r}} = \myvec{z_{s+r}}$ would imply that $\myvec{e_r}$ is in the row space of $B$).
	
	Since $[k] = P_1 \cup \cdots \cup P_{2s}$ is a partition, we may define $\myvec{y_1},\ldots,\myvec{y_k} \in \{\myvec{z_1},\ldots,\myvec{z_{2s}}\} \subseteq S$ in such a way that $\myvec{y_j} = \myvec{z_r}$ if and only if $j \in P_r$.
	Then for all $i \in [m]$ we have
	\begin{align*}
		a_{i1} \myvec{y_1} + \cdots + a_{ik} \myvec{y_k} &= \sum_{j \in P_1} a_{ij} \myvec{z_1} + \cdots + \sum_{j \in P_{2s}} a_{ij} \myvec{z_{2s}} \\
		&= b_{i1} \myvec{z_1} + \cdots + b_{is} \myvec{z_s} \, - \, b_{i1} \myvec{z_{s+1}} - \cdots - b_{is} \myvec{z_{2s}} = 0,
	\end{align*}
	so $(\myvec{y_1},\ldots,\myvec{y_k}) \in S^k$ is a solution of \mytag{}.
	Clearly $|\{\myvec{y_1},\ldots,\myvec{y_k}\}| \geq 1$.
	Furthermore, by assumption~\ref{itm:mod:stronger:eq}, for every column equivalence class $C = \{j_1,j_2\}$ of size $2$ that sums to zero, there is some $r \in [s]$ such that $P_r = \{j_1\}$ and $P_{s+r} = \{j_2\}$, so it follows that $\myvec{y_{j_1}} = \myvec{z_r} \neq \myvec{z_{s+r}} = \myvec{y_{j_2}}$.
\end{proof}

\section{Preliminaries on temperate systems}
\label{sec:temperate-prelims}
We now shift our attention from moderate to temperate systems.
We show that the problem of finding a generic solution is closely related to the problem of finding solutions of high dimension, and we show that we may once again restrict our attention to irreducible systems.

For an affine subspace $X \subseteq \F_q^n$ we let $\dim(X)$ denote the dimension of $X$. So $\dim(X)$ is the maximum number of affinely independent vectors in $X$ minus one. For a set $S \subseteq \F_q^n$, we let $\aff(S)$ denote the affine hull of $S$.

\begin{definition}
	For any given $k$-tuple $(\myvec{x_1},\ldots,\myvec{x_k}) \in (\F_q^n)^k$, let
	\[ \Annbal(\myvec{x_1},\ldots,\myvec{x_k}) = \{(b_1,\ldots, b_k) \in \F_q^k \mid b_1\myvec{x_1} + \cdots + b_k\myvec{x_k} = 0, \quad b_1 + \cdots + b_k = 0\}. \]
	So the elements of $\Annbal(\myvec{x_1},\ldots,\myvec{x_k})$ correspond to the balanced linear equations satisfied by $(\myvec{x_1},\ldots,\myvec{x_k})$.
\end{definition}

\begin{lemma}
	\label{lem:affine-rank-nullity}
	For every $(\myvec{x_1},\ldots,\myvec{x_k}) \in (\F^n)^k$ we have
	\[ \dim(\aff(\myvec{x_1},\ldots,\myvec{x_k})) + \dim(\Annbal(\myvec{x_1},\ldots,\myvec{x_k})) = k - 1. \]
\end{lemma}
\begin{proof}
	Let $A \in \F^{(n + 1) \times k}$ be the matrix
	\[ A = \begin{pmatrix} 1 & \cdots & 1 \\[1.5ex] \vert & & \vert \\ \myvec{x_1} & \cdots & \myvec{x_k} \\ \vert & & \vert \end{pmatrix}. \]
	For $I \subseteq [k]$ the vectors $\myvec{x_i}, \ i\in I$ are affinely independent if and only if the columns of $A$ indexed by $I$ are linearly independent. So $\rank(A) = \dim(\aff(\myvec{x_1},\ldots,\myvec{x_k})) + 1$.
	
	Evidently, $\ker(A)$ is precisely $\Annbal(\myvec{x_1},\ldots,\myvec{x_k})$, so the result follows from the rank-nullity theorem.
\end{proof}

\begin{corollary}
	\label{cor:affine-rank-nullity}
	Let \mytag{} be a balanced linear system of rank $m$, with coefficient matrix $A \in \F_q^{m \times k}$, and let $(\myvec{x_1},\ldots,\myvec{x_k})$ be a solution of \mytag{}.
	Then $\dim(\aff(\myvec{x_1},\ldots,\myvec{x_k})) \leq k - m - 1$, with equality if and only if $(\myvec{x_1},\ldots,\myvec{x_k})$ is a generic solution of \mytag{}.
\end{corollary}
\begin{proof}
	Since $(\myvec{x_1},\ldots,\myvec{x_k})$ is a solution of the system \mytag{}, the row space of $A$ is contained in $\Annbal(\myvec{x_1},\ldots,\myvec{x_k})$.
	Therefore we have $m = \rank(A) \leq \dim(\Annbal(\myvec{x_1},\ldots,\myvec{x_k}))$, so it follows from \cref{lem:affine-rank-nullity} that
	\[ \dim(\aff(\myvec{x_1},\ldots,\myvec{x_k})) = k - 1 - \dim(\Annbal(\myvec{x_1},\ldots,\myvec{x_k})) \leq k - 1 - m. \]
	Clearly we have equality if and only if the row space of $A$ is equal to $\Annbal(\myvec{x_1},\ldots,\myvec{x_k})$, which is equivalent to saying that all balanced linear equations satisfied by $(\myvec{x_1},\ldots,\myvec{x_k})$ are linear combinations of the equations in \mytag{}.
\end{proof}

\begin{proposition}
	\label{prop:temperate-reduce}
	Suppose that \mytag{} is equivalent to a linear system \mytagprime{} whose coefficient matrix $A'$ can be written as
	\[ A'=\begin{pmatrix} A_1 & 0 \\ 0 & A_2 \end{pmatrix} \]
	for some $A_1 \in \F_q^{m_1 \times k_1}$ and $A_2 \in \F_q^{m_2 \times k_2}$ with $m_1,m_2,k_1,k_2 \neq 0$.
	Then \mytag{} is temperate if and only if the systems given by $A_1$ and $A_2$ are temperate.
\end{proposition}
\begin{proof}
	If $\mytagprime{}$ is temperate, then it is easy to see that the same holds for the systems given by $A_1$ and $A_2$.
	
	Suppose that for $i=1,2$ the system given by $A_i$ is temperate, with constants $\beta_i,\gamma_i > 0$, where $\gamma_i< q$. Let $\gamma $ satisfy $\max(\gamma_1,\gamma_2)<\gamma<q$, and choose $\beta$ such that
	\[ \beta q^{\gamma n}\geq \max(qn\cdot \beta_1q^{\gamma_1n},\: nq^{k_1}\cdot \beta_2q^{\gamma_2n}) \qquad \text{for all $n \in \Z_{\geq 1}$}. \]
	Let $S\subseteq \F_q^n$ have size $|S|\geq \beta q^{\gamma n}$. For $i\in [n]$ and $\alpha\in \F_q$, write $S(i,\alpha):=\{\myvec{x} \in S \mid x_i = \alpha\}$.
	We claim that there exist $i\in [n]$ and distinct $\alpha',\alpha''\in \F_q$ such that $|S(i,\alpha')|, |S(i,\alpha'')|\geq \tfrac{|S|}{qn}$.
	For each coordinate $i\in[n]$, let $\alpha_i\in \argmax_{\alpha\in \F_q} |S(i,\alpha)|$ be a most popular value.
	Then $S\setminus\{(\alpha_1,\ldots, \alpha_n)\} = \cup_{i\in [n]} (S\setminus S(i,\alpha_i))$.
	So we can choose $i\in [n]$ such that $|S\setminus S(i,\alpha_i)| \geq \tfrac{|S|-1}{n}$.
	Then there is an $\alpha''\neq \alpha_i$ such that $S(i,\alpha'')\geq \frac{|S|-1}{n(q-1)}\geq \frac{|S|}{qn}$.
	Taking $\alpha'=\alpha_i$ proves the claim.
	
	Without loss of generality, we will assume that we can take $i=1$ in the claim. We denote
	$S_1 = S(1,\alpha')$ and $S_2 = S(1,\alpha'')$.
	Since $|S_1|\geq \beta_1q^{\gamma_1n}$, there exists a generic solution $\myvecvec{y} = (\myvec{y_1},\ldots,\myvec{y_{k_1}})\in (S_1)^{k_1}$ to the linear system given by $A_1$. We can take $I\subseteq [n]$ with $|I|\leq k_1-1$ such that for all $\myvec{b} = (b_1,\ldots, b_{k_1})\in \F_q^{k_1}$ with $b_1+\cdots+b_{k_1}=0$ we have:
	\begin{equation*}
		\forall i\in I: (b_1\myvec{y_1} + \cdots + b_{k_1}\myvec{y_{k_1}})_i = 0 \implies b_1\myvec{y_1} + \cdots + b_{k_1}\myvec{y_{k_1}} = 0.
	\end{equation*}
	Indeed, if $M\in \F_q^{n\times k_1}$ is the matrix with columns $\myvec{y_1},\ldots,\myvec{y_{k_1}}$, then we can take $I\subseteq [n]$ of size $|I|\leq k_1 - 1$ such that the rows of $M$ are contained in the span of the rows indexed by $I$ and the row vector $(1,\ldots, 1)$. Since $y_{11} = \cdots = y_{k_11}$ we may assume that $1\not\in I$.
	
	As $\myvecvec{y}$ is a generic solution to the system given by $A_1$, we obtain
	\begin{equation}\label{eqn:genericcoordinates}
		\forall i\in I: (b_1\myvec{y_1} + \cdots + b_{k_1}\myvec{y_{k_1}})_i = 0 \implies \myvec{b} \in \rowspace(A_1).
	\end{equation}
	We can take $\alpha_i\in \F_q$ for each $i\in I$ such that $T=\{\myvec{x} \in S_2 \mid x_i = \alpha_i \ \text{for all $i\in I$}\}$ has size $|T| \geq |S_2|\cdot q^{1-k_1} \geq \beta_2 q^{\gamma_2n}$.
	
	It follows that there exists a generic solution $\myvecvec{z} \in T^{k_2}$ to the system given by $A_2$.
	Now $\myvecvec{x} = (\myvecvec{y},\myvecvec{z})$ is a generic solution to \mytagprime{}.
	Indeed, let $\myvec{b} = (b_1,\ldots, b_k)\in \Annbal(\myvec{x_1},\ldots,\myvec{x_k})$.
	It suffices to show that $\myvec{b} \in \rowspace(A')$.
	Looking at the first coordinate and using that $b_1 + \cdots + b_{k} = 0$, we see that
	\begin{equation*}
		0 = (b_1 + \cdots + b_{k_1}) \alpha' + (b_{k_1+1} + \cdots + b_k) \alpha'' = (b_1 + \cdots + b_{k_1}) (\alpha'-\alpha'').
	\end{equation*}
	Since $\alpha'\neq \alpha''$, we find that $b_1 + \cdots + b_{k_1} = 0 = b_{{k_1}+1} + \cdots + b_k$.
	Since $\myvecvec{z} \in T^{k_2}$ it follows that
	\[ (b_1\myvec{y_1} + \cdots + b_{k_1}\myvec{y_{k_1}})_i = (b_1\myvec{x_1} + \cdots + b_k\myvec{x_k})_i = 0 \ (\forall i\in I). \]
	It now follows from \eqref{eqn:genericcoordinates} that $(b_1,\ldots, b_{k_1})\in \rowspace(A_1)$. So after modifying $\myvec{b}$ by an element of $\rowspace(A')$, we may assume that $b_1,\ldots,b_{k_1} = 0$. Hence the fact that $\myvec{b} \in \Annbal(\myvec{x_1},\ldots,\myvec{x_k})$ implies that $b_{k_1+1}\myvec{z_1} + \cdots + b_k\myvec{z_{k_2}} = 0$. Since $\myvecvec{z}$ is generic, we conclude that $(b_{k_1+1},\ldots,b_k)\in \rowspace(A_2)$. Hence, $\myvec{b} \in \rowspace(A')$.
\end{proof}

\section{Proof of \autoref{thm:intro:temperate} and \autoref{thm:intro:rank}}
\label{sec:temperate-proofs}

In this section, we develop the multiple replacement trick (\cref{cor:multiple-replacement}) and use it (in combination with \cref{lem:matrix-pigeonhole}) to prove \cref{thm:intro:temperate} and \cref{thm:intro:rank}.

We start with a many-solutions version of \cref{lem:pairs}.

\begin{lemma}
	\label{lem:manypairs}
	Let $q$ be a prime power, let $N_0 = (\Gamma_q)^n$, where $\Gamma_q$ is as in \cref{thm:tricoloured}, and let $t$ be a positive integer.
	Let $\myvec{x_1},\ldots,\myvec{x_L} \in \F_q^n$ be distinct, let $\myvec{y_1},\ldots,\myvec{y_L} \in \F_q^n$ be distinct, and let $\alpha,\beta\in \F_q\setminus\{0\}$.
	If $L\geq 4tN_0$, then there exists an $i \in [L]$ such that
	\[ \left| \big\{(i',i'') \in ([L] \setminus \{i\})^2 \mid\alpha \myvec{x_{i'}} + \beta \myvec{y_{i''}} = \alpha \myvec{x_i} + \beta \myvec{y_i} \big\} \right| \geq t. \]
\end{lemma}
\begin{proof}
	Write
	\[ T := \{(i,i',i'')\in [L]^3\mid \alpha \myvec{x_{i'}} + \beta \myvec{y_{i''}} = \alpha \myvec{x_{i}} + \beta \myvec{y_{i}} \ \text{and} \  i \neq i',i''\}. \]
	By \cref{lem:pairs}, the set $T \cap J^3$ is nonempty for all $J \subseteq [L]$ with $|J| \geq N_0$.
	We claim that $|T \cap J^3| \geq |J| - N_0$ for all $J\subseteq [L]$.
	Indeed, suppose that $|T \cap J^3| < |J| - N_0$; then we could delete fewer than $|J| - N_0$ elements from $J$ to obtain a set $J'$ of size $|J'| > N_0$ such that $T \cap (J')^3$ is empty: a contradiction.
	So $|T \cap J^3| - |J| + N_0 \geq 0$ for all $J \subseteq [L]$.
	
	Let $J$ be the random subset of $[L]$ obtained by independently taking each element of $[L]$ with probability $\tfrac{1}{2t}$.
	We have $\E[|J|] = \tfrac{L}{2t}$ and $\E[|T \cap J^3|] \leq \tfrac{|T|}{(2t)^2}$ since $|\{i,i',i''\}| \geq 2$ for all $(i,i',i'') \in T$.
	From $\E[|T\cap J^3| - |J| + N_0] \geq 0$ we obtain $\tfrac{|T|}{4t^2} \geq \frac{L}{2t} - N_0$, and therefore $\tfrac{|T|}{L} \geq 2t - \frac{4t^2N_0}{L} \geq t$.
	Hence, by the pigeonhole principle, there is an $i \in [L]$ such that $|\{(i',i'') \in [L]^2 \mid (i,i',i'') \in T\}| \geq t$, as required.
\end{proof}

Recall that two solutions $(\myvec{x_1},\ldots,\myvec{x_k})$ and $(\myvec{y_1},\ldots,\myvec{y_k})$ are said to be disjoint if $\{\myvec{x_1},\ldots,\myvec{x_k}\} \cap \{\myvec{y_1},\ldots,\myvec{y_k}\} = \varnothing$. We obtain a corollary analogous to \cref{cor:single-replacement}.
\begin{corollary}[`Multiple replacement trick']
	\label{cor:multiple-replacement}
	Let $\{(\myvec{x_1^{(i)}},\ldots,\myvec{x_k^{(i)}})\}_{i=1}^L$ be a list of pairwise disjoint solutions of \mytag{}, and suppose that $j_1$ and $j_2$ are distinct indices from the same column equivalence class.
	Suppose that $L \geq 4t\cdot (\Gamma_q)^n$ for some positive integer $t$. Then there exists $i\in [L]$ and $t$ distinct pairs $(i'_s,i''_s)\in ([L] \setminus \{i\})^2$, $s\in [t]$, such that $(\myvec{y^{(s)}_1},\ldots,\myvec{y^{(s)}_k})\in (\F_q^n)^k$ given by
	\[ \myvec{y^{(s)}_j} = \begin{cases}
		\myvec{x_j^{(i)}},&\quad\text{if $j \neq j_1,j_2$};\\[1ex]
		\myvec{x_j^{(i_s')}},&\quad\text{if $j = j_1$};\\[1ex]
		\myvec{x_j^{(i''_s)}},&\quad\text{if $j = j_2$};
	\end{cases} \]
	is also a solution of \mytag{} for all $s\in [t]$.
\end{corollary}
\begin{proof}
	Since the $j_1$-th and $j_2$-th column of \mytag{} are nonzero multiples of one another, we may choose a vector $\myvec{v} \in \F_q^m$ and constants $\alpha,\beta \neq 0$ such that the $j_1$-th column is equal to $\alpha \myvec{v}$ and the $j_2$-th column is equal to $\beta \myvec{v}$.
	
	By assumption, the vectors $\myvec{x_{j_1}^{(1)}},\ldots,\myvec{x_{j_1}^{(L)}}$ are pairwise distinct, and likewise the vectors $\myvec{x_{j_2}^{(1)}},\ldots,\myvec{x_{j_2}^{(L)}}$ are pairwise distinct, so it follows from \cref{lem:manypairs} that there exist $i\in [L]$ and $t$ distinct pairs $(i'_s,i''_s)\in ([L] \setminus \{i\})^2$, $s\in [t]$, with $\alpha \myvec{x_{j_1}^{(i)}} + \beta \myvec{x_{j_2}^{(i)}} = \alpha \myvec{x_{j_1}^{(i_s')}} + \beta \myvec{x_{j_2}^{(i_s'')}}$.
	Hence, the total contribution of $\myvec{x_{j_1}^{(i)}}$ and $\myvec{x_{j_2}^{(i)}}$ to the equations of \mytag{} is the same as the contribution of $\myvec{x_{j_1}^{(i_s')}}$ and $\myvec{x_{j_2}^{(i_s'')}}$.
	Since $(\myvec{x_1^{(i)}},\ldots,\myvec{x_k^{(i)}})$ is a solution of \mytag{}, so is $(\myvec{y^{(s)}_1},\ldots,\myvec{y^{(s)}_k})$.
\end{proof}

\begin{definition}
	Let $A\in \F_q^{m\times k}$ be a matrix and let $j_1,j_2\in [k]$ be distinct elements in the same column equivalence class of $A$. We say that $(b_1,\ldots, b_k)\in \F_q^k$ \emph{breaks} the pair $\{j_1,j_2\}$ if after adding the row $(b_1,\ldots, b_k)$ to $A$, the columns indexed by $j_1$ and $j_2$ are no longer scalar multiples of one another.
\end{definition}

\begin{lemma}
	\label{lem:breakingpairs}
	Let \mytag, $A$, $m$, $k$ and $\ell$ be as in \cref{sit:system}, let $j_1,j_2\in [k]$ be distinct indices in the same column equivalence class, and let $\{(\myvec{x_1^{(i)}},\ldots,\myvec{x_k^{(i)}})\}_{i=1}^L$ be a list of pairwise disjoint solutions to \mytag{}.
	If $L \geq 4q^k(\Gamma_q)^n$, then there exists $i\in [L]$ and a solution $(\myvec{y_1},\ldots,\myvec{y_k})$ to \mytag{} such that:
	\begin{enumerate}[label=\textup{(\roman*)},beginpenalty=100]
		\item\label{itm:bp:recombine} $\myvec{y_j} = \myvec{x^{(i)}_j}$ for all $j\neq j_1,j_2$ and $\myvec{y_j} \in \{\myvec{x^{(1)}_j},\ldots,\myvec{x^{(L)}_j}\}$ for $j \in \{j_1,j_2\}$;
		\item\label{itm:bp:annbal} $\Annbal(\myvec{y_1},\ldots,\myvec{y_k}) \subseteq \Annbal(\myvec{x_1^{(i)}},\ldots,\myvec{x_k^{(i)}})$;
		\item\label{itm:bp:no-break} no $\myvec{b} \in \Annbal(\myvec{y_1},\ldots,\myvec{y_k})$ breaks the pair $\{j_1,j_2\}$.
	\end{enumerate}
\end{lemma}
\begin{proof}
	By \cref{cor:multiple-replacement}, we may choose $i\in [L]$ and a sequence $\{(i'_s,i''_s)\}_{s=1}^{q^k}$ of $q^k$ pairwise distinct pairs $(i_s',i_s'') \in ([L] \setminus \{i\})^2$ such that, for all $s \in [q^k]$, the $k$-tuple $(\myvec{z_1^{(s)}},\ldots,\myvec{z_k^{(s)}}) \in S^k$ defined by
	\[
		\myvec{z^{(s)}_j} = \begin{cases}
			\myvec{x^{(i)}_j} & \text{if $j\in [k]\setminus\{j_1,j_2\}$}\\[1ex]
			\myvec{x^{(i'_s)}_{j}} & \text{if $j=j_1$}\\[1ex]
			\myvec{x^{(i''_s)}_{j}} & \text{if $j=j_2$}
		\end{cases}
	\]
	is a solution to \mytag{}.
	
	If $\myvec{b} = (b_1,\ldots, b_k)$ breaks the pair $\{j_1,j_2\}$, then the contributions $b_{j_1}\myvec{z^{(s)}_{j_1}} + b_{j_2}\myvec{z^{(s)}_{j_2}}$ for $s \in [q^k]$ are pairwise distinct.
	Therefore we can have $\myvec{b} \in \Annbal(\myvec{z_1^{(s)}},\ldots,\myvec{x_k^{(s)}})$ for at most one value of $s$.
	Since the number of $\myvec{b} \in \F_q^k$ with $b_1 + \cdots + b_k = 0$ is less than $q^k$, we may choose $s_0 \in [q^k]$ such that no $\myvec{b} \in \Annbal(\myvec{z_1^{(s_0)}},\ldots,\myvec{z_k^{(s_0)}})$ breaks the pair $\{j_1,j_2\}$.
	
	Set $y := \myvec{z^{(s_0)}}$.
	Then \ref{itm:bp:recombine} and \ref{itm:bp:no-break} are met.
	To prove \ref{itm:bp:annbal}, let $\myvec{b} \in \Annbal(\myvec{y_1},\ldots,\myvec{y_k})$ be given.
	Since $\myvec{b}$ does not break the pair $\{j_1,j_2\}$, we have $b_{j_1}\myvec{z^{(s_0)}_{j_1}} + b_{j_2}\myvec{z^{(s_0)}_{j_2}} = b_{j_1}\myvec{x^{(i)}_{j_1}} + b_{j_2}\myvec{x^{(i)}_{j_2}}$, and therefore $\myvec{b} \in \Annbal(\myvec{x_1^{(i)}},\ldots,\myvec{x_k^{(i)}})$, as desired.
\end{proof}

\begin{lemma}
	\label{basecase}
	Let \mytag, $A$, $m$, $k$ and $\ell$ be as in \cref{sit:system}. Let $S\subseteq \F_q^n$ have size $|S|\geq q^{1+\tfrac{\ell-1}{\ell}n}$. Assume that at least one of the following two conditions holds:
	\begin{enumerate}[label=\textup{(\roman*)},beginpenalty=100]
		\item $\ell = m + 1$;
		\item every column equivalence class sums to zero.
	\end{enumerate}
	Then there exists a solution $\myvecvec{x} = (\myvec{x_1},\ldots,\myvec{x_k})\in S^k$ to \mytag{} with the following property:
	\begin{equation}
		\label{niceproperty}
		\begin{gathered}
			\text{If $\myvec{b} \in \Annbal(\myvec{x_1},\ldots,\myvec{x_k})$ preserves the column}\\
			\text{equivalence classes of \mytag{}, then $\myvec{b} \in \rowspace(A)$.}
		\end{gathered}
	\end{equation}
\end{lemma}
\begin{proof}
	Let $[k] = C_1\cup\cdots\cup C_\ell$ be the partition of $[k]$ into column equivalence classes.
	
	We first consider the case that condition (i) holds. Let $\myvecvec{x} = (\myvec{x_1},\ldots,\myvec{x_k})$ be any solution to~\mytag{}.
	Suppose that $\myvecvec{x}$ satisfies a balanced equation $b_1\myvec{x_1} + \cdots + b_k\myvec{x_k} = 0$ that preserves the column equivalence classes of \mytag, but $(b_1,\ldots, b_k)$ is not in the row space of $A$.
	Let $A'$ be the $(m+1)\times k$ matrix obtained by adding the row $(b_1,\ldots, b_k)$ to $A$. Then $\rank(A') = m+1 = \ell$.
	For $t\in [\ell]$ let $\myvec{\sigma_t} \in \F_q^{m+1}$ be the sum of the columns of $A'$ in class $C_t$.
	Since the column rank of $A'$ is $\ell$, it follows that if we take one index from each column equivalence class, the corresponding $\ell$ columns are linearly independent.
	Let $I = \{t \in [\ell] \mid \myvec{\sigma_t} \neq 0\}$.
	Then the $\myvec{\sigma_t}$, $t\in I$ are linearly independent and $\sum_{t\in I} \myvec{\sigma_t} = \sum_{t\in [\ell]} \myvec{\sigma_t} = 0$.
	It follows that $I = \varnothing$.
	So all column equivalence classes of $A'$ (and hence of $A$) sum to zero, and we are in case (ii).
	
	We now consider the case that condition (ii) holds. Denote by $V\subseteq \F_q^k$ the set of vectors that preserve the column equivalence classes of $A$. We will assume (by reordering the columns of $A$) that $C_1=\{1,\ldots, |C_1|\},\ldots, C_\ell=\{k-|C_\ell|,\ldots, k\}$. So there are row vectors $\myvec{v}_t\in \F_q^{|C_t|}$ such that 
	\[ V = \{\begin{bmatrix}c_1\myvec{v}_1 &\cdots&c_\ell\myvec{v}_\ell\end{bmatrix} \mid c_1,\ldots, c_\ell\in \F_q\}. \]
	Since the rows of $A$ belong to $V$ and $A$ has no zero columns, the $\myvec{v}_t$ have only nonzero entries. By scaling, we may assume that the first entry of $\myvec{v}_t$ equals $1$. For $t\in [\ell]$ let $j_t=|C_1|+\cdots+|C_{t-1}|+1$. So for all $\myvec{b}\in V$, we have $\myvec{b}=\begin{bmatrix}b_{j_1}\myvec{v}_1 &\cdots&b_{j_\ell}\myvec{v}_\ell\end{bmatrix}$.
	
	Let $A'=(a'_{it})\in \F_q^{m\times \ell}$ be the submatrix of $A$ induced by columns $j_1,\ldots, j_t$. Then $(b_1,\ldots, b_k)$ is in the row space of $A$ if and only if $(b_{j_1},\ldots, b_{j_\ell})$ is in the row space of $A'$. 
	
	Consider the system
	\[ \sum_{t=1}^\ell a'_{it}\myvec{y_t} = 0\text{ for all $i\in [m]$}. \]
	Since $|S|\geq q^{1+\frac{\ell-1}{\ell}n}$, it follows by \cref{lem:matrix-pigeonhole} that there are $(\myvec{y_1},\ldots,\myvec{y_\ell})$ and $(\myvec{z_1},\ldots,\myvec{z_\ell})$ in $S^\ell$ such that for all $(b_1,\ldots, b_\ell)\in \F_q^\ell$ one has $b_1(\myvec{y_1} - \myvec{z_1}) + \cdots + b_\ell(\myvec{y_\ell} - \myvec{z_\ell}) = 0$ if and only if $(b_1,\ldots, b_\ell)$ is in the row space of $A'$. Define $(\myvec{x_1},\ldots,\myvec{x_k})\in S^k$ by setting (for $t\in [\ell]$ and $j\in C_t$)
	\[ \myvec{x_{j}} = \begin{cases}
		\myvec{y_t} & \text{if $j = j_t$},\\
		\myvec{z_t} & \text{if $j\in C_t\setminus\{j_t\}$.}
	\end{cases} \]

	Since the entries of each $\myvec{v}_t$ sum to zero (the column equivalence classes sum to zero by assumption), we have 
	\begin{equation}\label{eq:restriction}
		b_{j_1}(\myvec{y}_1-\myvec{z}_1)+\cdots +b_{j_\ell}(\myvec{y}_\ell-\myvec{z}_\ell)=0\iff b_1\myvec{x}_1+\cdots+b_k\myvec{x}_k=0
	\end{equation}
	for every $\myvec{b}\in V$.

	We now check that $(\myvec{x}_1,\ldots, \myvec{x}_k)$ satisfies the required properties. To show that it is a solution to \mytag{}, let $\myvec{b}$ be a row of $A$. Then the restriction $(b_{j_1},\ldots, b_{j_\ell})$ is in the row space of $A'$, so $b_{j_1}(\myvec{y}_1-\myvec{z}_1)+\cdots +b_{j_\ell}(\myvec{y}_\ell-\myvec{z}_\ell)=0$. Hence, by \eqref{eq:restriction}, we have $b_1\myvec{x}_1+\cdots+b_k\myvec{x}_k=0$ as required.

	Let $\myvec{b}\in V\cap \Annbal(\myvec{x}_1,\ldots, \myvec{x}_k)$. It remains to show that $\myvec{b}$ is in the row space of $A$. By \eqref{eq:restriction}, we have $b_{j_1}(\myvec{y}_1-\myvec{z}_1)+\cdots +b_{j_\ell}(\myvec{y}_\ell-\myvec{z}_\ell)=0$, so $(b_{j_1},\ldots, b_{j_\ell})$ is in the row space of $A'$. It follows that $\myvec{b}$ is in the row space of $A$.
\end{proof}

We are now ready to prove \cref{thm:intro:temperate} and \cref{thm:intro:rank}.

\begin{proof}[Proof of \cref{thm:intro:temperate}]
	Let $\Gamma_q$ be the constant from \cref{thm:tricoloured}.
	For every nonnegative integer $t$, we define
	\[ N_t := q^{1+\tfrac{\ell-1}{\ell}n}+t\cdot(4kq^{k}(\Gamma_q)^n). \]
	Let $[k] = C_1\cup\cdots\cup C_\ell$ be the partition of $[k]$ into column equivalence classes of $A$.
	We will prove by induction on $|P|$ that, for every set $P \subseteq \binom{C_1}{2}\cup\cdots\cup \binom{C_\ell}{2}$ of equivalent pairs and for every set $S \subseteq \F_q^n$ of size $|S| \geq N_{|P|}$, the system \mytag{} has a solution $\myvecvec{x} = (\myvec{x_1},\ldots,\myvec{x_k}) \in S^k$ that satisfies \eqref{niceproperty} and such that no $(b_1,\ldots,b_k) \in \Annbal(\myvec{x_1},\ldots,\myvec{x_k})$ breaks a pair in $P$.
	\begin{itemize}
		\item For $|P| = 0$, the claim follows directly from \cref{basecase}.
		
		\item Assume that $|P|\geq 1$ and that the claim holds for all sets of fewer than $|P|$ pairs.
		Fix some $\{j_1,j_2\}\in P$, write $L = 4q^k(\Gamma_q)^n$, and let $S \subseteq \F_q^n$ be a set of size $|S| \geq N_{|P|}$.
		Since $|S| \geq N_{|P|} \geq kL + N_{|P|-1}$, it follows from the induction hypothesis that there exist $L$ pairwise disjoint solutions $\myvecvecnum{x}{(1)},\ldots,\myvecvecnum{x}{(L)} \in S^k$ to \mytag{} that satisfy \eqref{niceproperty} and such that no $(b_1,\ldots,b_k) \in \Annbal(\myvec{x_1^{(i)}},\ldots,\myvec{x_k^{(i)}})$ breaks a pair in $P \setminus \{\{j_1,j_2\}\}$, for all $i \in [L]$.
		
		By \cref{lem:breakingpairs}, we may choose $i_0 \in [L]$ and a solution $\myvecvec{x} = (\myvec{x_1},\ldots,\myvec{x_k})\in S^k$ to \mytag{} such that $\Annbal(\myvec{x_1},\ldots,\myvec{x_k}) \subseteq \Annbal(\myvec{x_1^{(i_0)}},\ldots,\myvec{x_k^{(i_0)}})$ and no $\myvec{b} \in \Annbal(\myvec{x_1},\ldots,\myvec{x_k})$ breaks the pair $\{j_1,j_2\}$.
		By construction, $\myvecvecnum{x}{(i_0)}$ satisfies \eqref{niceproperty} and no $\myvec{b} \in \Annbal(\myvec{x_1^{(i_0)}},\ldots,\myvec{x_k^{(i_0)}})$ breaks a pair in $P\setminus\{\{j_1,j_2\}\}$, so the same properties are true for $\myvecvec{x}$, because $\Annbal(\myvec{x_1},\ldots,\myvec{x_k})\subseteq \Annbal(\myvec{x_1^{(i_0)}},\ldots,\myvec{x_k^{(i_0)}})$.
		We conclude that $\myvecvec{x}$ satisfies \eqref{niceproperty} and no $\myvec{b} \in \Annbal(\myvec{x_1},\ldots,\myvec{x_k})$ breaks a pair in $P$.
	\end{itemize}
	Letting $P = \binom{C_1}{2}\cup\cdots\cup \binom{C_\ell}{2}$ completes the proof.
\end{proof}

\begin{proof}[Proof of \cref{thm:intro:rank}]
	If all column equivalence classes sum to zero, the result follows directly from \mycref{thm:intro:temperate}{itm:intro:tmp:z}.
	Assume therefore that not all column equivalence classes sum to zero.
	Let $\Gamma_q$ be the constant from \cref{thm:tricoloured}.
	For every nonnegative integer $t$ we define
	\[ N_t:=t\cdot(4kq^{k}(\Gamma_q)^n). \]
	Let $S\subseteq \F_q^n$ have size $|S|\geq N_{k_2}$. By the same argument as in the proof of \cref{thm:intro:temperate}, we have a solution $\myvecvec{x} = (\myvec{x_1},\ldots,\myvec{x_k}) \in S^k$ to \mytag{} such that no $\myvec{b} \in \Annbal(\myvec{x_1},\ldots,\myvec{x_k})$ breaks a pair from the same column equivalence class.
	In other words, $\myvec{b}$ preserves the column equivalence classes, so this proves part \ref{itm:intro:rnk:preserve}.
	
	For part \ref{itm:intro:rnk:dim-aff}, observe that $\Annbal(\myvec{x_1},\ldots,\myvec{x_k})$ does not contain all balanced linear equations that preserve the column equivalence classes, for otherwise every column equivalence class must sum to zero, contrary to our assumption.
	So we have $\dim(\Annbal(\myvec{x_1},\ldots,\myvec{x_k})) \leq \ell - 1$, and therefore $\dim(\aff(\myvec{x_1},\ldots,\myvec{x_k}))\geq k - \ell$, by \cref{lem:affine-rank-nullity}.
\end{proof}

\begin{remark}
	\label{rmk:rank-improvement}
	We compare the rank of the solution $(\myvec{x_1},\ldots,\myvec{x_k})$ in \cref{thm:intro:rank} to the rank given by \cref{thm:intro:Sauermann-rank}.
	Suppose we are in \cref{sit:system}, and set $r = k - 2m + 1$.
	Then $k \geq 2m - 1 + r$, so it follows from \cref{thm:intro:Sauermann-rank} that we can find a solution with $\dim(\spn(\myvec{x_1},\ldots,\myvec{x_k})) \geq r$, and therefore $\dim(\aff(\myvec{x_1},\ldots,\myvec{x_k})) \geq r - 1 = k - 2m$.
	
	So how do these two compare?
	If $\ell = 1$, then we must have $m = 1$ (because we assume that the rows of $A$ are linearly independent), so in this case the rank from \cref{thm:intro:rank} and \cref{thm:intro:Sauermann-rank} agree.
	If $\ell \geq 2$, then we see that \cref{thm:intro:rank} improves upon \cref{thm:intro:Sauermann-rank} whenever $m > \frac{\ell}{2}$.
	Then again, \cref{thm:intro:rank} only applies to a smaller class of linear systems.
\end{remark}

\section{Examples and applications}
\label{sec:examples}

We conclude this paper by looking at a few examples of type (RC) linear systems, to highlight the applications and limitations of the results from this paper.
First we will look at an application to sumsets in $\F_q^n$.
We show that our results can be used to find non-trivial solutions of an arbitrary linear system in the difference set $S - S$, but not in the sumset $S + S$.
After that, we will look at the systems studied by Mimura and Tokushige \cite{Mimura-Tokushige-star,Mimura-Tokushige-shape,Mimura-Tokushige-II}.
We show that our techniques furnish alternative proofs that those systems are moderate, and in many cases we strengthen this to show that the system is also temperate.

\subsection{Applications to sum and difference sets}
\label{subsec:examples:sumsets}

Since this paper studies linear systems with repeated columns, one obvious question is to which extent our results can be applied to the problem of finding solutions to a system of linear equations in sum and difference sets.
Throughout this section, let $\F_q$ be a finite field of characteristic $p$, and let $c_1,\ldots,c_l \in \F_q \setminus \{0\}$.
We consider the affinely independent sumset (or AIR-sumset)
\[ T := c_1 \cdot S \AIRplus \cdots \AIRplus c_l \cdot S = \{c_1 \myvec{x_1} + \cdots + c_l \myvec{x_l} \, \mid \, \myvec{x_1},\ldots,\myvec{x_l} \in S\ \text{affinely independent}\}. \]
If $c_1 + \cdots + c_l = 0$, then \cref{cor:intro:application-general} states that $T$ contains generic solutions to every linear system \mytag{}, provided that $S$ is sufficiently large.
We now prove this statement.

\begin{proof}[Proof of \cref{cor:intro:application-general}]
	Let $A = (a_{ij}) \in \F_q^{m \times k}$ be the coefficient matrix of the system \mytag{}.
	(Recall from the statement of \cref{cor:intro:application-general} that $A$ may be arbitrary.)
	Let $A' = (a_{ij}') \in \F_q^{m \times lk}$ be the $m \times lk$ matrix
	\[ A' = \begin{bmatrix} c_1 A \mid c_2 A \mid \cdots \mid c_l A \end{bmatrix}, \]
	and let \mytagprime{} be the corresponding linear system.
	Every column equivalence class of \mytagprime{} is the union of sets of the form $\{j,j+k,\ldots,j+(l-1)k\}$ (for some $j \in [k]$), so \mytagprime{} is of type (RC).
	Furthermore, the column equivalence classes sum to zero, because $c_1 + \cdots + c_l = 0$.
	Hence it follows from \mycref{thm:intro:temperate}{itm:intro:tmp:z} and \cref{prop:temperate-reduce} that \mytagprime{} is temperate.
	Therefore there are constants $\beta,\gamma \geq 1$ with $\gamma < q$ such that every set $S \subseteq \F_q^n$ with $|S| \geq \beta \cdot \gamma^n$ contains a generic solution of~\mytagprime{}.
	Choose such a generic solution $(\myvec{x_1},\ldots,\myvec{x_{lk}}) \in S^{lk}$, and define $\myvec{y_1},\ldots,\myvec{y_k} \in c_1\cdot S + \cdots + c_l\cdot S$ by
	\[ \myvec{y_j} := c_1 \myvec{x_j} + c_2 \myvec{x_{j+k}} + \cdots + c_l \myvec{x_{j + (l-1)k}}. \]
	Clearly $(\myvec{y_1},\ldots,\myvec{y_k})$ is a solution of the linear system \mytag{}.
	We show that $(\myvec{y_1},\ldots,\myvec{y_k})$ is linearly generic and that $\myvec{y_1},\ldots,\myvec{y_k} \in (c_1\cdot S \AIRplus \cdots \AIRplus c_l \cdot S) \cup \{0\}$.
	
	First, let $\myvec{b} = (b_1,\ldots,b_k) \in \F_q^k$ be such that $b_1 \myvec{y_1} + \cdots + b_k \myvec{y_k} = 0$.
	Then $(\myvec{x_1},\ldots,\myvec{x_{lk}})$ belongs to the kernel of the $1 \times lk$ matrix
	\[ B' = \begin{bmatrix} c_1 \myvec{b} \mid c_2 \myvec{b} \mid \cdots \mid c_l \myvec{b} \end{bmatrix}. \]
	Since $c_1 + \cdots + c_l = 0$, the entries of $B'$ sum to $0$, so $B'$ represents a balanced linear equation satisfied by $(\myvec{x_1},\ldots,\myvec{x_{lk}})$.
	Since $(\myvec{x_1},\ldots,\myvec{x_{lk}})$ is a generic solution of \mytagprime{}, it follows that $B'$ is in the row space of $A'$.
	Equivalently, $\myvec{b}$ is in the row space of $A$.
	This shows that $(\myvec{y_1},\ldots,\myvec{y_k})$ is linearly generic.
	
	To complete the proof, it suffices to show that $\myvec{y_j} = 0$ whenever the vectors $\myvec{x_j},\myvec{x_{j+k}},\ldots,\myvec{x_{j+(l-1)k}}$ are affinely dependent, for every $j \in [k]$.
	To that end, suppose that $\myvec{x_j},\myvec{x_{j+k}},\ldots,\myvec{x_{j+(l-1)k}}$ are affinely dependent.
	Then there is some $\myvec{b} = (b_1,\ldots,b_l) \in \F_q^l \setminus \{0\}$ with $b_1 + \cdots + b_l = 0$ and
	\begin{equation*}
		b_1 \myvec{x_j} + b_2 \myvec{x_{j+k}} + \cdots + b_l \myvec{x_{j + (l-1)k}} = 0. \tag*{$(\myvec{b'})$}
	\end{equation*}
	Since $(\myvec{x_1},\ldots,\myvec{x_{lk}})$ is generic, the balanced linear equation $(\myvec{b'})$ is a linear combination of the equations in \mytagprime{}.
	By choosing some $r \in [l]$ such that $b_r \neq 0$ and restricting our attention to the variables $\myvec{x_{(r-1)k+1}},\ldots,\myvec{x_{rk}}$ (i.e.{} the $r$-th block in the block matrix representation of $A'$), we see that the equation $\myvec{y_j} = 0$ is a linear combination of the equations in \mytag{}.
\end{proof}

\Cref{cor:intro:application-AP} can be deduced from \cref{cor:intro:application-general} by letting \mytag{} be the linear system that encodes a $k$-term arithmetic progression and setting $l = 2$ and $(c_1,c_2) = (1,-1)$.
We show that \cref{cor:intro:application-AP} does not depend on the full strength of \cref{cor:intro:application-general}, as it follows immediately from \cref{lem:matrix-pigeonhole}.

\begin{proof}[Proof of \cref{cor:intro:application-AP}]
	Let \mytag{} be a linear system which encodes a $k$-term arithmetic progression, for instance the system given by the matrix
	\[ A = \begin{pmatrix}
		1 & -2 & 1 & 0 & 0 & \cdots & 0 & 0 & 0 & 0 & 0 \\
		0 & 1 & -2 & 1 & 0 & \cdots & 0 & 0 & 0 & 0 & 0 \\
		\vdots & \vdots & \vdots & \vdots & \vdots & \ddots & \vdots & \vdots & \vdots & \vdots & \vdots \\
		0 & 0 & 0 & 0 & 0 & \cdots & 0 & 1 & -2 & 1 & 0 \\
		0 & 0 & 0 & 0 & 0 & \cdots & 0 & 0 & 1 & -2 & 1
	\end{pmatrix} \, \in \, \F_p^{(k - 2) \times k}. \]
	Let $S \subseteq \F_p^n$ with $|S| \geq p^{1 + (1 - \frac{1}{k})n}$.
	By \cref{lem:matrix-pigeonhole}, there are $(\myvec{x_1},\ldots,\myvec{x_k}),(\myvec{y_1},\ldots,\myvec{y_k}) \in S^k$ such that $(\myvec{x_1} - \myvec{y_1} , \ldots, \myvec{x_k} - \myvec{y_k})$ is a linearly generic solution of \mytag{}.
	
	Since the standard basis vectors $\myvec{e_1},\ldots,\myvec{e_k} \in \F_q^k$ are not in the row space of $A$,\hair\savefootnote{\hasAPs}{To prove this, it is sufficient to note that there exist non-trivial $k$-APs in $\F_q^n \setminus \{0\}$.} we have $\myvec{x_j} - \myvec{y_j} \neq 0$ for all $j \in [k]$.
	Likewise, since the vectors $\myvec{e_j} - \myvec{e_{j'}}$ ($j \neq j'$) are not in the row space of $A$,\hair\reusefootnotemark{\hasAPs} we have $\myvec{x_j} - \myvec{y_j} \neq \myvec{x_{j'}} - \myvec{y_{j'}}$ whenever $j \neq j'$.
	It follows that $(\myvec{x_1} - \myvec{y_1} , \ldots, \myvec{x_k} - \myvec{y_k})$ is a non-trivial $k$-AP in $(S - S) \setminus \{0\}$.
\end{proof}

\begin{remark}
	The preceding proof carries through unchanged if $A$ is replaced by an arbitrary matrix, and if the difference set $(S - S) \setminus \{0\}$ is replaced by the sum set $c_1 \cdot S + \cdots + c_l \cdot S$ with $c_1 + \cdots + c_l = 0$ (replace $\myvec{x_j} - \myvec{y_j} \in S - S$ by $c_1\myvec{x_j} + (c_2 + \cdots + c_l)\myvec{y_j} \in c_1 \cdot S + \cdots + c_l \cdot S$).
	So a weaker version of \cref{cor:intro:application-general}, where the AIR-sumset is replaced by an ordinary sumset, can also be proved by a simple counting argument, without using the slice rank method.
\end{remark}

\begin{remark}
	Now consider once again the sumset $c_1\cdot S + \cdots + c_l \cdot S$, but this time assume that $c_1 + \cdots + c_l \neq 0$.
	In this case, the techniques from this paper do not say anything non-trivial about the problem of finding a non-trivial $k$-AP in the sum set $c_1 \cdot S + \cdots + c_l \cdot S$.
	
	We explain why the results from this paper do not work when $c_1 + \cdots + c_l \neq 0$.
	It is tempting to try to repeat the proof of \cref{cor:intro:application-general}, but we run into a problem: The column equivalence classes no longer sum to zero, so we have to replace \mycref{thm:intro:temperate}{itm:intro:tmp:z} by \mycref{thm:intro:temperate}{itm:intro:tmp:nz}.
	However, this imposes two extra conditions on the original $m \times k$ matrix in the proof of \cref{cor:intro:application-general}, namely that $A \one = 0$ (i.e. \mytag{} is balanced) and that $k = \rank(A) + 1$.
	So we can only say something for a very specific class of linear systems.
	In fact, this class is so specific that the coefficient matrix must satisfy $\ker(A) = \spn(\one)$, so every solution of the original system must be constant!
	
	Likewise, it is tempting to try to repeat the proof of \cref{cor:intro:application-general}, but this time replacing \mycref{thm:intro:temperate}{itm:intro:tmp:z} by \mycref{thm:intro:moderate}{itm:intro:mod:nz}.
	After all, to find (say) a non-trivial $k$-AP, it is enough to find a solution with $\myvec{y_1},\ldots,\myvec{y_k}$ pairwise distinct instead of a generic solution.
	Here we run into another problem.
	In the proof of \cref{cor:intro:application-general}, we can find a solution $(\myvec{x_1},\ldots,\myvec{x_{lk}}) \in S^{lk}$ of the extended system \mytagprime{} with $\myvec{x_1},\ldots,\myvec{x_{lk}}$ pairwise distinct.
	But when we recombine these to form a solution $(\myvec{y_1},\ldots,\myvec{y_k}) \in (c_1 \cdot S + \cdots + c_l \cdot S)^k$ of the original system \mytag{}, we may end up with $\myvec{y_1} = \cdots = \myvec{y_k}$, since we have no way to avoid these additional equations.
	In fact, if we use the proof of \mycref{thm:intro:moderate}{itm:intro:mod:nz} as an algorithm to find the $\myvec{x_1},\ldots,\myvec{x_{lk}}$, then this is guaranteed to happen: We start with a solution where all variables $\myvec{x_1},\ldots,\myvec{x_{lk}}$ are equal, and then modify the variables in such a way that the contribution to each column equivalence class remains the same, so the equation $\myvec{y_1} = \cdots = \myvec{y_k}$ is maintained throughout the proof.
	Once again, the techniques from this paper are unable to say anything non-trivial.
\end{remark}

\subsection{The systems studied by Mimura and Tokushige}
\label{subsec:examples:Mimura-Tokushige}
In a series of papers \cite{Mimura-Tokushige-star,Mimura-Tokushige-shape,Mimura-Tokushige-II}, Mimura and Tokushige studied several specific (classes of) linear systems, and showed that each of them is moderate.
These were the first results of this type.
We show that our results and techniques furnish alternative proofs for all systems studied by Mimura and Tokushige (though our constants might not be as good).

The systems studied by Mimura and Tokushige have integer entries, and can therefore be interpreted as a linear system over $\F_q$ for an arbitrary prime power $q = p^s$.
Depending on the system, Mimura and Tokushige sometimes had to assume that $p \neq 2$ or $p \neq 3$, and we shall do the same.

\begin{example}
	\label{xmpl:MT:star}
	In \cite{Mimura-Tokushige-star}, Mimura and Tokushige studied a \emph{star of $k$ three-term arithmetic progressions}, given by the linear system $(\mathcal S_{*k})$ with coefficient matrix
	\[ \begin{pmatrix}
		1 & 1 & 0 & 0 & \cdots & 0 & 0 & -2 \\
		0 & 0 & 1 & 1 & \cdots & 0 & 0 & -2 \\
		\vdots & \vdots & \vdots & \vdots & \ddots & \vdots & \vdots & \vdots \\
		0 & 0 & 0 & 0 & \cdots & 1 & 1 & -2
	\end{pmatrix} \, \in \, \F_q^{k \times (2k + 1)}, \]
	and proved that this system is moderate whenever $p \geq 3$.
	
	This result can be recovered as a special case of \cref{thm:intro:moderate}, and strengthened to $(\mathcal S_{*k})$ being temperate by \cref{thm:intro:temperate}.
	Indeed, $(\mathcal S_{*k})$ is a type (RC) linear system, as it is balanced and there is only one column equivalence class of size $1$.
	If $p \neq 2$, then the system is non-degenerate and irreducible, and all column equivalence classes have sum $\pm 2 \neq 0$, so it follows from \mycref{thm:intro:moderate}{itm:intro:mod:nz} that $(\mathcal S_{*k})$ is moderate.
	Additionally, since there are $k$ equations and $k + 1$ column equivalence classes, it follows from \mycref{thm:intro:temperate}{itm:intro:tmp:nz} that $(\mathcal S_{*k})$ is temperate.
\end{example}

\begin{example}
	\label{xmpl:MT:fan}
	Also in \cite{Mimura-Tokushige-star}, Mimura and Tokushige point out that their proof also extends to a `fan' of $k$ three-term arithmetic progressions, given by the linear system $(\mathcal S_{*k}')$ with coefficient matrix
	\[ \begin{pmatrix}
		1 & -2 & 0 & 0 & \cdots & 0 & 0 & 1 \\
		0 &  0 & 1 & -2 & \cdots & 0 & 0 & 1 \\
		\vdots & \vdots & \vdots & \vdots & \ddots & \vdots & \vdots & \vdots \\
		0 &  0 & 0 & 0 & \cdots & 1 & -2 & 1
	\end{pmatrix} \, \in \, \F_q^{k \times (2k + 1)}. \]
	Analogously to \cref{xmpl:MT:star}, it follows from \mycref{thm:intro:moderate}{itm:intro:mod:nz} and \mycref{thm:intro:temperate}{itm:intro:tmp:nz} that $(\mathcal S_{*k}')$ is moderate and temperate, provided that $p \neq 2$.
\end{example}

\begin{example}
	\label{xmpl:MT:W}
	In \cite{Mimura-Tokushige-shape}, Mimura and Tokushige studied the problem of avoiding a `W shape', and showed that the linear system $(\mathcal W)$ with coefficient matrix
	\[ \begin{pmatrix}
		1 & -1 & -1 & 1 & 0 \\
		1 & 0 & -2 & 0 & 1
	\end{pmatrix} \, \in \, \F_q^{2 \times 5} \]
	is moderate whenever $p \geq 3$.
	
	This is not a type (RC) linear system, since there are $3$ column equivalence classes of size $1$, so this result cannot be recovered as a special case of \cref{thm:intro:moderate} or \cref{thm:intro:temperate}.
	
	Nevertheless, our techniques from \mysecref{sec:mod:nz} can be adapted to recover this result as well.
	Indeed, let $\Gamma_q$ be the constant from \cref{thm:tricoloured}, and let $S \subseteq \F_q^n$ with $|S| \geq 4\cdot (\Gamma_q)^n$.
	By repeatedly finding a non-trivial $3$-AP and removing it from $S$, we can find a list $\{(\myvec{x_1^{(i)}},\myvec{x_3^{(i)}},\myvec{x_5^{(i)}})\}_{i=1}^L$ of $L \geq (\Gamma_q)^n$ pairwise disjoint non-trivial $3$-APs in $S^3$.
	For all $i \in [L]$, set $\myvec{x_2^{(i)}} = \myvec{x_3^{(i)}}$ and $\myvec{x_4^{(i)}} = \myvec{x_5^{(i)}}$, so that $(\myvec{x_1^{(i)}},\myvec{x_2^{(i)}},\myvec{x_3^{(i)}},\myvec{x_4^{(i)}},\myvec{x_5^{(i)}}) \in S^5$ is a solution of $(\mathcal W)$.
	Since $2$ and $4$ belong to the same column equivalence class, it follows from \cref{cor:single-replacement} that there are $i \neq i',i''$ such that the $5$-tuple $(\myvec{y_1},\myvec{y_2},\myvec{y_3},\myvec{y_4},\myvec{y_5}) = (\myvec{x_1^{(i)}},\myvec{x_2^{(i')}},\myvec{x_3^{(i)}},\myvec{x_4^{(i'')}},\myvec{x_5^{(i)}}) \in S^5$ is also a solution of $(\mathcal W)$.
	Then $\myvec{y_1},\myvec{y_3},\myvec{y_5}$ are pairwise distinct because they stem from the same non-trivial $3$-AP, and $\{\myvec{y_1},\myvec{y_3},\myvec{y_5}\} \cap \{\myvec{y_2},\myvec{y_4}\} = \varnothing$ because they stem from disjoint solutions.
	Finally, note that $\myvec{y_2} \neq \myvec{y_4}$, for otherwise the first equation of $(\mathcal W)$ would imply that $\myvec{y_1} = \myvec{y_3}$.
	This shows that $(\mathcal W)$ is moderate.
	
	With minor modifications, the preceding argument also shows that $(\mathcal W)$ is temperate.
	Indeed, by repeating the argument, but using multiple replacement (\cref{cor:multiple-replacement}) instead of single replacement (\cref{cor:single-replacement}), we can make sure that $\myvec{x_2^{(i')}}$ is not in the line through $\myvec{x_1^{(i)}}$, $\myvec{x_3^{(i)}}$ and $\myvec{x_5^{(i)}}$.
	Then $\dim(\aff(\myvec{x_1^{(i)}},\myvec{x_2^{(i')}},\myvec{x_3^{(i)}},\myvec{x_4^{(i'')}},\myvec{x_5^{(i)}})) \geq 2$, so it follows from \cref{cor:affine-rank-nullity} that this solution is generic.
\end{example}

\begin{example}
	\label{xmpl:MT:T}
	In \cite{Mimura-Tokushige-II}, Mimura and Tokushige studied the system $(T)$ with coefficient matrix
	\[ \begin{pmatrix}
		1 & -2 & 1 & 0 & 0 \\
		0 & 0 & -2 & 1 & 1
	\end{pmatrix} \, \in \, \F_q^{2 \times 5}, \]
	and proved that it is moderate whenever $p \geq 3$.
	
	Once again, this result can be recovered as a special case of \mycref{thm:intro:moderate}{itm:intro:mod:nz}, and strengthened to $(T)$ being temperate by \mycref{thm:intro:temperate}{itm:intro:tmp:nz}.
\end{example}

\begin{example}
	\label{xmpl:MT:lSk+2}
	In \cite{Mimura-Tokushige-II}, Mimura and Tokushige studied the class of linear systems $(l S_{k+2})$.
	This class is defined as follows: let $k \geq 1$, and let $a_1,\ldots,a_{k+2} \in \F_q$ be non-zero such that $a_1 + \cdots + a_{k+2} = 0$.
	Then $(l S_{k+2})$ is given by the coefficient matrix
	\[ \begin{pmatrix}
		a_1 & \cdots & a_k & a_{k+1} & a_{k+2} & 0 & 0 & \cdots & 0 & 0 \\
		a_1 & \cdots & a_k & 0 & 0 & a_{k+1} & a_{k+2} & \cdots & 0 & 0 \\
		\vdots & \ddots & \vdots & \vdots & \vdots & \vdots & \vdots & \ddots & \vdots & \vdots \\
		a_1 & \cdots & a_k & 0 & 0 & 0 & 0 & \cdots & a_{k+1} & a_{k+2}
	\end{pmatrix} \, \in \, \F_q^{l \times (k + 2l)}. \]
	In \cite[Thm.{} 5]{Mimura-Tokushige-II}, Mimura and Tokushige showed that such a system is always moderate.
	(This contains the linear system $(S_1)$ from \cite{Mimura-Tokushige-II} as a special case.)
	
	This result can be recovered as a special case of \cref{thm:intro:moderate}, and strengthened to $(l S_{k+2})$ being temperate by \cref{thm:intro:temperate}.
	Indeed, $(l S_{k+2})$ is balanced, and it has one column equivalence class of size $k \geq 1$ and $l$ column equivalence classes of size $2$, so it is a type (RC) linear system.
	Furthermore, the system is non-degenerate and irreducible.
	Note that, if one column equivalence class sums to zero, then all column equivalence classes must sum to zero, so it follows from either \mycref{thm:intro:moderate}{itm:intro:mod:nz} or \mycref{thm:intro:moderate}{itm:intro:mod:z} that $(l S_{k+2})$ is moderate.
	Furthermore, since the number of equations is $l$ and the number of column equivalence classes is $l + 1$, it follows from either \mycref{thm:intro:temperate}{itm:intro:tmp:nz} or \mycref{thm:intro:temperate}{itm:intro:tmp:z} that $(l S_{k+2})$ is temperate.
\end{example}

\begin{example}
	\label{xmpl:MT:2Tkl}
	In \cite{Mimura-Tokushige-II}, Mimura and Tokushige studied the class of linear systems $(2T_{k,l})$.
	This class is defined as follows: let $k \geq 1$ and $l \geq 2$, and let $a_1,\ldots,a_{k+l} \in \F_q$ be non-zero such that $a_1 + \cdots + a_{k+l} = 0$.
	Then $(2T_{k,l})$ is given by the coefficient matrix
	\[ \begin{pmatrix}
		a_1 & \cdots & a_k & a_{k+1} & \cdots & a_{k+l} & 0 & \cdots & 0 \\
		a_1 & \cdots & a_k & 0 & \cdots & 0 & a_{k+1} & \cdots & a_{k+l}
	\end{pmatrix} \, \in \, \F_q^{2 \times (k + 2l)}. \]
	In \cite[Thm.{} 6]{Mimura-Tokushige-II}, Mimura and Tokushige showed that such a system is always moderate.
	(This contains the linear system $(S_2)$ from \cite{Mimura-Tokushige-II} as a special case.)
	
	This result can be recovered as a special case of \cref{thm:intro:moderate}, and strengthened to $(2T_{k,l})$ being temperate by \cref{thm:intro:temperate}.
	The argument is analogous to that of \cref{xmpl:MT:lSk+2}.
\end{example}

\begin{example}
	\label{xmpl:MT:S3-}
	In \cite{Mimura-Tokushige-II}, Mimura and Tokushige studied the linear system $(S_3^-)$ with coefficient matrix
	\[ \begin{pmatrix}
		1 & 1 & 1 & 1 & -4 & 0 & 0 & 0 & 0 & 0 \\
		1 & 1 & 0 & 0 & 0 & 1 & 1 & -4 & 0 & 0 \\
		1 & 1 & 0 & 0 & 0 & 1 & 0 & 0 & 1 & -4
	\end{pmatrix} \, \in \, \F_q^{3 \times 10}, \]
	and proved that it is moderate whenever $p \neq 2$.\hair\footnote{The authors don't make the assumption $p \neq 2$ explicit in their proof. This assumption is necessary because the sum of the second and third row of the coefficient matrix is congruent to $\begin{pmatrix} 0 & 0 & 0 & 0 & 0 & 0 & 1 & 0 & -1 & 0 \end{pmatrix} \pmod 2$. So for $p = 2$ the system cannot be moderate because it forces two variables to be equal.}
	
	This result can be recovered as a special case of \cref{thm:intro:moderate}, provided that $p \neq 2,3$.\hair\footnote{If $p = 2$, then there are three column equivalence classes of size $1$, so the system is not of type (RC). Furthermore, if $p \in \{2,3\}$, then there are column equivalence classes of size $2$ that sum to zero, but not all column equivalence classes sum to $0$, so neither \mycref{thm:intro:moderate}{itm:intro:mod:nz} nor \mycref{thm:intro:moderate}{itm:intro:mod:z} applies in this case. If $p \notin \{2,3\}$, then the system is of type (RC) and none of column equivalence classes sums to zero, so \mycref{thm:intro:moderate}{itm:intro:mod:nz} applies.}
	The results from this paper are insufficient to determine whether $(S_3^-)$ is temperate, because there are not enough equations to apply \mycref{thm:intro:temperate}{itm:intro:tmp:nz}.
\end{example}

\begin{example}
	\label{xmpl:MT:S3}
	Finally, in \cite[Conjecture 1]{Mimura-Tokushige-II}, Mimura and Tokushige conjectured that the system $(S_3)$ with coefficient matrix
	\[ \begin{pmatrix}
		1 & 1 & 1 & 1 & -4 & 0 & 0 & 0 & 0 & 0 & 0 \\
		1 & 1 & 0 & 0 & 0 & 1 & 1 & -4 & 0 & 0 & 0 \\
		1 & 1 & 0 & 0 & 0 & 0 & 0 & 0 & 1 & 1 & -4
	\end{pmatrix} \, \in \, \F_q^{3 \times 11} \]
	is moderate.
	This is confirmed by our results. If $p \neq 2$, then it follows from \mycref{thm:intro:moderate}{itm:intro:mod:nz} and \mycref{thm:intro:temperate}{itm:intro:tmp:nz} that $(S_3)$ is moderate and temperate.
	If $p = 2$, then some of the columns become zero, so they correspond to free variables. After removing those columns, it follows from \mycref{thm:intro:moderate}{itm:intro:mod:z} and \mycref{thm:intro:temperate}{itm:intro:tmp:z} that $(S_3)$ is moderate and temperate.
\end{example}

In summary: in all examples except \cref{xmpl:MT:S3-}, we were able to prove that the system is moderate and temperate, thereby strengthening prior results (and proving a conjecture) of Mimura and Tokushige.
In \cref{xmpl:MT:S3-}, we gave an alternative proof of the fact that the system is moderate, but we were unable to determine whether the system is also temperate.

In \cref{xmpl:MT:W}, we could not apply \cref{thm:intro:moderate}.
Instead, we needed a proof that was adapted to this particular system, using results from \mysecref{sec:mod:nz}, to furnish an alternative proof that the system is moderate.
In all other examples, the fact that the system is moderate follows immediately from \cref{thm:intro:moderate}.

\small
\paragraph{Acknowledgements}
The first author is partially supported by the Dutch Research Council (NWO), project number 613.009.127.

\bibliographystyle{alphaurl}
\bibliography{repeated-columns}

\newcommand{\etalchar}[1]{$^{#1}$}
\providecommand{\noopsort}[1]{} \providecommand{\DobbendeBruyn}{van Dobben de
  Bruyn} \providecommand{\Do}{Dob}
\begin{thebibliography}{BCC{\etalchar{+}}17}

\bibitem[BCC{\etalchar{+}}17]{Blasiak-et-al}
Jonah Blasiak, Thomas Church, Henry Cohn, Joshua~A. Grochow, Eric Naslund,
  William~F. Sawin, and Chris Umans.
\newblock On cap sets and the group-theoretic approach to matrix
  multiplication.
\newblock {\em Discrete Analysis}, 2017:3, 2017.
\newblock 27pp.
\newblock \href {https://doi.org/10.19086/da.1245}
  {\path{doi:10.19086/da.1245}}.

\bibitem[Bou90]{Bourgain-AP}
J.~Bourgain.
\newblock On arithmetic progressions in sums of sets of integers.
\newblock In A.~Baker, B.~Bollob{\'a}s, and A.~Hajnal, editors, {\em A tribute
  to Paul Erd{\H{o}}s}, pages 105--110. Cambridge University Press, 1990.
\newblock \href {https://doi.org/10.1017/CBO9780511983917.008}
  {\path{doi:10.1017/CBO9780511983917.008}}.

\bibitem[CLP17]{Croot-Lev-Pach}
Ernie Croot, Vsevolod~L. Lev, and P\'eter~P\'al Pach.
\newblock {Progression-free sets in \(\mathbb{Z}_4^n\) are exponentially
  small}.
\newblock {\em Annals of Mathematics}, 185(1):331--337, 2017.
\newblock \href {https://doi.org/10.4007/annals.2017.185.1.7}
  {\path{doi:10.4007/annals.2017.185.1.7}}.

\bibitem[\Do23]{Dobben-dissertation}
Josse \DobbendeBruyn{}.
\newblock {\em Divisorial gonality of graphs, the slice rank polynomial method,
  and tensor products of convex cones}.
\newblock PhD thesis, TU Delft, 2023.
\newblock \href
  {https://doi.org/10.4233/uuid:bb2db244-e032-46bd-a9d7-a36b9ce0ce0e}
  {\path{doi:10.4233/uuid:bb2db244-e032-46bd-a9d7-a36b9ce0ce0e}}.

\bibitem[EG17]{Ellenberg-Gijswijt}
Jordan~S. Ellenberg and Dion Gijswijt.
\newblock {On large subsets of \(\mathbb{F}_q^n\) with no three-term arithmetic
  progression}.
\newblock {\em Annals of Mathematics}, 185(1):339--343, 2017.
\newblock \href {https://doi.org/10.4007/annals.2017.185.1.8}
  {\path{doi:10.4007/annals.2017.185.1.8}}.

\bibitem[MT19a]{Mimura-Tokushige-star}
Masato Mimura and Norihide Tokushige.
\newblock Avoiding a star of three-term arithmetic progressions,
  {\noopsort{2019a}}2019.
\newblock Preprint.
\newblock URL: \url{https://arxiv.org/abs/1909.10507}.

\bibitem[MT19b]{Mimura-Tokushige-shape}
Masato Mimura and Norihide Tokushige.
\newblock Avoiding a shape, and the slice rank method for a system of
  equations, {\noopsort{2019b}}2019.
\newblock Preprint.
\newblock URL: \url{https://arxiv.org/abs/1909.10509}.

\bibitem[MT20]{Mimura-Tokushige-II}
Masato Mimura and Norihide Tokushige.
\newblock Solving linear equations in a vector space over a finite field {II},
  2020.
\newblock Preprint.
\newblock URL: \url{http://www.cc.u-ryukyu.ac.jp/~hide/sol2.pdf}.

\bibitem[Nas20]{Naslund-EGZ-constant}
Eric Naslund.
\newblock Exponential bounds for the {E}rd{\H{o}}s--{G}inzburg--{Z}iv constant.
\newblock {\em Journal of Combinatorial Theory, Series A}, 174:105185, 2020.
\newblock \href {https://doi.org/10.1016/j.jcta.2019.105185}
  {\path{doi:10.1016/j.jcta.2019.105185}}.

\bibitem[Ruz93]{Ruzsa-I}
Imre~Z. Ruzsa.
\newblock Solving a linear equation in a set of integers {I}.
\newblock {\em Acta Arithmetica}, 65(3):259--282, 1993.
\newblock \href {https://doi.org/10.4064/aa-65-3-259-282}
  {\path{doi:10.4064/aa-65-3-259-282}}.

\bibitem[Ruz95]{Ruzsa-II}
Imre~Z. Ruzsa.
\newblock Solving a linear equation in a set of integers {II}.
\newblock {\em Acta Arithmetica}, 72(4):385--397, 1995.
\newblock \href {https://doi.org/10.4064/aa-72-4-385-397}
  {\path{doi:10.4064/aa-72-4-385-397}}.

\bibitem[Sau21]{Sauermann-zero}
Lisa Sauermann.
\newblock On the size of subsets of \(\mathbb{F}_p^n\) without \(p\) distinct
  elements summing to zero.
\newblock {\em Israel Journal of Mathematics}, 243(1):63--79, 2021.
\newblock \href {https://doi.org/10.1007/s11856-021-2145-x}
  {\path{doi:10.1007/s11856-021-2145-x}}.

\bibitem[Sau23]{Sauermann-systems}
Lisa Sauermann.
\newblock {Finding solutions with distinct variables to systems of linear
  equations over \(\mathbb{F}_p^n\)}.
\newblock {\em Mathematische Annalen}, 386(1-2):1--33, 2023.
\newblock \href {https://doi.org/10.1007/s00208-022-02391-y}
  {\path{doi:10.1007/s00208-022-02391-y}}.

\bibitem[Tao16]{Tao-slice-rank}
Terence Tao.
\newblock Notes on the ``slice rank'' of tensors, 2016.
\newblock Blog post.
\newblock URL:
  \url{https://terrytao.wordpress.com/2016/08/24/notes-on-the-slice-rank-of-tensors/}.

\bibitem[TV06]{Tao-Vu}
Terence Tao and Van~H. Vu.
\newblock {\em Additive Combinatorics}, volume 105 of {\em Cambridge Studies in
  Advanced Mathematics}.
\newblock Cambridge University Press, 2006.

\end{thebibliography}

\end{document}